\documentclass[11pt,a4paper]{amsart}
\usepackage{nicefrac}
\usepackage{geometry}
\usepackage{amssymb,amsmath,amsfonts,amsthm}

\usepackage[TS1,T1]{fontenc}    
\usepackage[utf8]{inputenc}

\usepackage{parskip}
\usepackage{mathabx} 
\usepackage{ifthen,xifthen}
\usepackage{paralist}
\usepackage[ngerman,english]{babel}

\usepackage{tikz}
\definecolor{darkred}{rgb}{0.5,0,0}
\definecolor{darkgreen}{rgb}{0,0.5,0}
\definecolor{darkblue}{rgb}{0,0,0.5}
\usepackage[capitalise]{cleveref}
\crefname{page}{page}{pages}
\crefname{property}{property}{properties}
\creflabelformat{property}{\upshape#2(#1)#3}
\crefname{propertyM}{property}{properties}
\creflabelformat{propertyM}{{\upshape#2(M#1)#3}}
\crefname{statement}{statement}{statements}
\creflabelformat{statement}{\itshape{#2(#1)#3}}

\numberwithin{equation}{section}    
\usepackage{aliascnt}           
\newcommand{\mynewtheorem}[4][]{
  \ifthenelse{\isempty{#1}}{    
    \newtheorem{#2}{#3}         
  }{
    \newaliascnt{#2}{#1}        
    \newtheorem{#2}[#2]{#3}     
    \aliascntresetthe{#2}       
  }
  \crefname{#2}{#3}{#4}         
}
\theoremstyle{plain}
\newtheorem{Theorem}{Theorem}[section]
\crefname{Theorem}{Theorem}{Theorems}
\mynewtheorem[Theorem]{Proposition}{Proposition}{Propositions}
\mynewtheorem[Theorem]{Corollary}{Corollary}{Corollaries}
\mynewtheorem[Theorem]{Lemma}{Lemma}{Lemmas}
\theoremstyle{definition}
\mynewtheorem[Theorem]{Definition}{Definition}{Definitions}
\mynewtheorem[Theorem]{Assumption}{Assumption}{Assumptions}
\mynewtheorem[Theorem]{Example}{Example}{Examples}
\mynewtheorem[Theorem]{Remark}{Remark}{Remarks}
\usepackage{wasysym}

\renewcommand{\epsilon}{\varepsilon}
\newcommand{\ve}{\varepsilon}
\renewcommand{\phi}{\varphi}

\DeclareMathOperator{\Tr}{Tr}

\DeclareMathOperator{\supp}{supp}

\providecommand{\abs}[1]{\lvert#1\rvert}
\providecommand{\bdry}{\partial}
\providecommand{\bigabs}[1]{\bigl\lvert#1\bigr\rvert}

\providecommand{\biggabs}[1]{\biggl\lvert#1\biggr\rvert}

\providecommand{\argmt}{\mathchoice{{}\cdot{}}{{}\cdot{}}{{}\bullet{}}{{}\bullet{}}}
\providecommand{\setsize}[1]{\lvert#1\rvert}
\providecommand{\bigsetsize}[1]{\bigl|#1\bigr|}
\providecommand{\Bigsetsize}[1]{\Bigl|#1\Bigr|}

\providecommand{\norm}[2][]{\lVert#2\rVert\ifthenelse{\isempty{#1}}{}{_{#1}}}
\providecommand{\bignorm}[2][]{\bigl\lVert#2\bigr\rVert\ifthenelse{\isempty{#1}}{}{_{#1}}}
\providecommand{\Bignorm}[2][]{\Bigl\lVert#2\Bigr\rVert\ifthenelse{\isempty{#1}}{}{_{#1}}}
\providecommand{\biggnorm}[2][]{\biggl\lVert#2\biggr\rVert\ifthenelse{\isempty{#1}}{}{_{#1}}}
\providecommand{\Biggnorm}[2][]{\Biggl\lVert#2\Biggr\rVert\ifthenelse{\isempty{#1}}{}{_{#1}}}
\providecommand{\floor}[1]{\lfloor#1\rfloor}
\providecommand{\Norm}[2][]{\left\lVert#2\right\rVert\ifthenelse{\isempty{#1}}{}{_{#1}}}

\providecommand{\spr}[3][]{{\langle#2,#3\rangle}\ifthenelse{\isempty{#1}}{}{_{#1}}}
\providecommand{\Spr}[3][]{\left\langle#2,#3\right\rangle\ifthenelse{\isempty{#1}}{}{_{#1}}}
\providecommand{\from}{\colon}
\providecommand{\dirac}[1]{\delta_{#1}}
\providecommand{\xto}{\xrightarrow}
\providecommand{\qtext}{\quad\text}
\providecommand{\qqtext}{\qquad\text}
\providecommand{\textq}[1]{\text{#1}\quad}
\providecommand{\qtextq}[1]{\quad\text{#1}\quad}
\providecommand{\Ioo}[3][]{\mathopen(#2,#3\mathclose)\ifthenelse{\isempty{#1}}{}{_{#1}}}
\providecommand{\Ico}[3][]{\mathopen[#2,#3\mathclose)\ifthenelse{\isempty{#1}}{}{_{#1}}}
\providecommand{\Ioc}[3][]{\mathopen(#2,#3\mathclose]\ifthenelse{\isempty{#1}}{}{_{#1}}}
\providecommand{\Icc}[3][]{\mathopen[#2,#3\mathclose]\ifthenelse{\isempty{#1}}{}{_{#1}}}
\providecommand{\ifu}[1]{\chi\ifthenelse{\isempty{#1}}{}{_{#1}}}
\providecommand{\isect}{\cap}

\providecommand{\union}{\cup}
\providecommand{\Union}{\bigcup}
\providecommand{\dunion}{\mathbin{\dot\cup}}
\providecommand{\dUnion}{\mathop{\dot\bigcup}}
\providecommand{\tensor}{\otimes}
\providecommand{\Tensor}{\bigotimes}
\providecommand{\Laplace}{\Delta}
\providecommand{\graph}{\Upsilon}

\providecommand{\cone}{\mathcal C}
\providecommand{\K}{K}
\let\originald\d 
\renewcommand{\d}{\ifthenelse{\boolean{mmode}}{\mathrm d}{\originald}}
\def\Int#1d{\int#1\,\d} 
\providecommand{\dx}{\d x}
\newcommand{\di}{k}
\newcommand{\RL}{\R^\di}

\newcommand{\B}{\mathbb{B}}

\newcommand{\E}{\mathbb{E}}
\newcommand{\N}{\mathbb{N}}
\newcommand{\PP}{\mathbb{P}}
\let\originalP\P 
\renewcommand{\P}{\ifthenelse{\boolean{mmode}}{\mathbb P}{\originalP}}

\newcommand{\R}{\mathbb{R}}
\newcommand{\RR}{\mathbb{R}}
\newcommand{\Z}{\mathbb{Z}}
\newcommand{\ZZ}{\mathbb{Z}}

\newcommand{\cA}{\mathcal{A}}
\newcommand{\cB}{\mathcal{B}}
\providecommand{\Borel}{\cB}

\newcommand{\cE}{\mathcal{E}}
\newcommand{\cF}{\mathcal{F}}

\newcommand{\cM}{\mathcal{M}}
\newcommand{\cN}{\mathcal{N}}

\newcommand{\cU}{\mathcal{U}}
\newcommand{\cZ}{\mathcal{Z}}
\newcommand{\fd}{\mathfrak{d}}

\providecommand{\density}{\rho}
\providecommand{\etdef}{\phantom:&\rlap:=}
\providecommand{\connectedto}[1][\omega]{\overset{#1}\leftrightsquigarrow}%
\providecommand{\cluster}{C}%
\providecommand{\M}{M}
\providecommand{\muc}{\mu_{\mathrm c}}%

\newcommand{\hm}[1]{\textbf{*}\leavevmode{\marginpar{\tiny%
$\hbox to 0mm{\hspace*{-0.5mm}$\leftarrow$\hss}%
\vcenter{\vrule depth 0.1mm height 0.1mm width \the\marginparwidth}%
\hbox to 0mm{\hss$\rightarrow$\hspace*{-0.5mm}}$\\\relax\raggedright #1}}}
\usepackage[final]{microtype}
\usepackage{lmodern}
\title[A Glivenko--Cantelli Theorem for almost additive functions]
{A Glivenko--Cantelli Theorem for almost additive functions on lattices}
\author{Christoph Schumacher, Fabian Schwarzenberger, Ivan Veseli\'c}
\address[C.~Schumacher, F.~Schwarzenberger, I.~Veseli\'c]
        {Fakult\"at f\"ur Mathematik,
        Technische Universit\"at Chemnitz,
        09107 Chemnitz, Germany}
\email{ivan.veselic@mathematik.tu-chemnitz.de}
\curraddr[F.~Schwarzenberger]
        {Fakult\"at Informatik/Mathematik, Hochschule f\"ur Technik und Wirtschaft Dresden, 01069 Dresden, Germany}
\date{2015-04-24}

\begin{document}

\maketitle

\begin{abstract}
  We develop a Glivenko--Cantelli theory for monotone, almost additive
  functions of i.\,i.\,d.\ sequences of random variables indexed by~$\Z^d$.
  Under certain conditions on the random sequence,
  short range correlations are allowed as well.
	We have an explicit error estimate,
	consisting of a probabilistic and a geometric part.
  We apply the results to yield uniform convergence
	for several quantities arising naturally in statistical physics.
\end{abstract}

\section{Introduction}

The classical Glivenko--Cantelli theorem states
that the empirical cumulative distribution functions of an increasing set
of independent and identically distributed random variables converge
\emph{uniformly} to the cumulative population distribution function almost surely.
Due to its importance to applications, e.\,g.\ statistical learning theory, the Glivenko--Cantelli theorem
is also called the ``fundamental theorem of statistics''.
The theorem has initiated the study of so-called Glivenko--Cantelli classes as they feature, for instance, in the Vapnik--Chervonenkis theory \cite{VapnikChervonenkis1971}.
Generalizations of the fundamental theorem rewrite the uniform convergence with respect to the real variable as
convergence of a \emph{supremum over a family (of sets or functions)} and widen the family over which the supremum is taken,
making the statement ``more uniform''.
However, there are limits to this uniformization:
For instance, if the original distribution is continuous,
there is no convergence if the supremum is taken
w.\,r.\,t.\ the family of finite subsets of the reals.
Thus, a balance has to be found between
the class over which the supremum is taken
and the distribution of the random variables,
the details of which are often dictated by the application in mind.
Another important extension are multivariate Glivenko--Cantelli theorems,
where the i.\,i.\,d.\ random variables are generalized to
i.\,i.\,d.\ random vectors with possibly dependent coordinates.
Such results have been obtained e.\,g.\ in
\cite{Rao1962,Stute1976,DeHardt1971,Wright1981}.
In contrast to the classical one-dimensional Glivenko--Cantelli theorem,
where no assumptions on the underlying distribution is necessary, in the higher dimensional case,
one has to exclude certain singular continuous measures, cf.~Theorem \ref{wright:LDP}.
The multidimensional version of the Portmanteau theorem provides a hint why such conditions are necessary.
We apply these results in \cref{GC}.

To avoid confusion, let us stress that uniform convergence in the classical Glivenko--Cantelli Theorem and in
our result involves discontinuous functions, so it is quite different to uniform convergence of differentiable functions,
as it is encountered e.g.~with power series.

In many models of statistical physics one shows that certain random
quantities are self-averaging, i.e.\ possess a well defined non-random thermodynamic limit.
This is not only true for random operators of Schr\"odinger type, cf.~e.g.~\cite{Spencer-86,PasturF-92,Veselic-08},
but also for spin  systems, cf.~e.g.~\cite{Griffiths-64,GriffithsL-68,Vuillermot-77,WehrA-90,Bovier-06}.
Note however that the latter  papers, studying the free energy (and derived quantities),
heavily use specific properties of the exponential function (entering the free energy) like
convexity and smoothness. We lack these properties in the Glivenko--Cantelli setting and are
thus dealing with a completely different situation.
The geometric ingredients of the proof of the thermodynamic limit can be traced back to papers by Van Hove \cite{VanHove-49}
and F\o lner \cite{Foelner-55}. This is why the exhaustion sets used in the thermodynamic limit are associated with their names.

While standard statistical problems concern i.\,i.\,d.\ samples,
an independence assumption quickly appears unnatural in statistical physics.
Neighboring entities in solid state models (such as atoms or spins)
are unlikely to not influence each other.
In order to treat physically relevant scenarios
one introduces a geometry to encode location and adjacency relations
between the random variables, which in turn are used to allow dependencies
between close random variables.
In the present paper we choose~$\Z^d$ as our model of physical space, 
although our methods should apply to amenable groups as well,
at least with an additional monotile condition.
The focus on $\Z^d$ allows us to avoid technicalities of amenable groups with monotiles
and can thus present our results in a simpler, more transparent manner.
Furthermore, we can achieve more explicit error bounds due to the simple 
geometry of~$\Z^d$.

Our main result is \cref{thm:main},
which is a Glivenko--Cantelli type theorem for a class of monotone,
almost additive functions and suitable distributions of the random variables,
allowing spatial dependencies.
Our precise hypotheses are spelled out in \cref{defP,def:admissible}.
The theorem can be interpreted as a multi-dimensional ergodic theorem
with values in the Banach space of right continuous and bounded functions
with $\sup$-norm, i.\,e.\ a uniform convergence result.
Under slightly strengthened assumptions we obtain an explicit error term
for the convergence,
which is a sum of a geometric and a probabilistic part, cf.~\cref{cor:main}.
While earlier Banach space valued ergodic theorems, e.\,g.\ \cite{LenzMV-08,LenzSV-10},
have been restricted to a finite set of colors,
we are able to treat the real-valued case.
To do this, we have to assume a monotonicity property,
which is satisfied in most cases of interest.
We obtain a more explicit convergence estimate than \cite{LenzMV-08}, as well.
This is due to the fact that we assume a short range correlation condition,
while \cite{LenzMV-08} assumes the existence of limiting frequencies.
The Glivenko--Cantelli result is applied to several examples
from statistical physics in \cref{sec:Ecf,sec:ccf}.
The flexibility and generality of our probabilistic model is displayed in
\cref{example:density}.

For the proof we use two sets of ideas.
The first one concerns geometric approximation and tiling arguments
for almost additive functions based on the amenability of the group $\Z^d$ going back to the mentioned seminal papers
of Van Hove \cite{VanHove-49} and F\o lner \cite{Foelner-55}.
In the context of Banach space valued ergodic theorems 
they have been used for instance in
\cite{Lenz-02,LenzS-05,LenzVeselic2009,LenzMV-08,LenzSV-10,PogorzelskiS-13}.
%
The second ingredient of the proof is multivariate Glivenko--Cantelli theory,
as developed in \cite{Rao1962,Stute1976,DeHardt1971,Wright1981}. Our \cref{strictlymonotone}
shows that in our setting a large deviations type estimate derived by Wright can be applied.
The latter is a modification of the Dvoretzky--Kiefer--Wolfowitz inequality \cite{DvoretzkyKieferWolfowitz1956,Massart1990}.

The structure of the paper is as follows:
In \cref{sec:notation} we present our notation and the two main theorems.
\cref{sec:wunsch} contains an intuitive sketch of the proof in the case $\Z^d=\Z$,
\cref{secboundary} geometric tiling and approximation arguments,
\cref{GC} multivariate Glivenko--Cantelli theory,
\cref{sec:Aadd} the proof of the main theorem, and
\cref{sec:Ecf,sec:ccf} examples.

\section{Notation and main results}\label{sec:notation}

The geometric setting of this paper is given via~$\Z^d$,
which gives in a natural way rise to a graph $(\Z^d,\cE)$.
Here, the set of edges~$\cE$ is the subset of the power set of~$\Z^d$,
consisting exactly of those $\{x,y\}\subseteq\Z^d$
which satisfy $\norm[1]{y-x}=1$.
As usual $\norm[1]{x}=\sum_{i=1}^d \abs{x_i}$ denotes the $\ell^1$-norm in $\Z^d$.
By~$\cF$ we denote the (countable) set which consists of all finite subsets of~$\Z^d$.
For $\Lambda\in\cF$, we write~$\setsize\Lambda$
for the number of elements in~$\Lambda$.
The metric on the set of vertices $\fd\from \Z^d\times \Z^d\to\N_0$ is defined via the $\ell^1$-norm,
i.\,e.\ for $x,y\in\Z^d$ we set $\fd(x,y):=\norm[1]{y-x}$.
For two sets $\Lambda_1,\Lambda_2\subseteq \Z^d$
we write $\fd(\Lambda_1,\Lambda_2):=\min\{\fd(x,y)\mid x\in \Lambda_1,y\in\Lambda_2\}$.
In the case that $\Lambda_1=\{x\}$ contains only one element
we write $\fd(x,\Lambda_2)$ for $\fd(\{x\},\Lambda_2)$.

For $\Lambda\subseteq \Z^d$ we write $\Lambda+z :=\{x+z\mid x\in \Lambda\}$.
A \emph{cube of side length $n\in\N$} is a set which is given by $
(\Ico0n^d\isect\Z^d)+z$ for some $z\in\Z^d$.

Using the metric~$\fd$, we define for $r\in\N_0$ the $r$-boundary
of a set $\Lambda\subseteq\Z^d$ by
\begin{equation*}
 \partial^r(\Lambda)
:=\{x\in \Lambda\mid \fd(x,\Z^d\setminus\Lambda)\le r\}
\union \{x\in \Z^d\setminus \Lambda \mid \fd(x,\Lambda)\le r\}\text.
\end{equation*}
Moreover, we set
\begin{equation}\label{eq:def:Lambda^r}
 \Lambda^r:=\Lambda\setminus \partial^r(\Lambda)
=\{x\in \Lambda\mid \fd(x,\Z^d\setminus \Lambda)> r\}\text.
\end{equation}
If $(\Lambda_n)_{n\in \N}$ (or short $(\Lambda_n)$)
is a sequence of subsets of~$\Z^d$,
we write $(\Lambda_n^r)_{n\in\N} $ or $(\Lambda_n^r)$
instead of $((\Lambda_n)^r)_{n\in\N}$.

Note that for a cube~$\Lambda_n$ of side length~$n$ and $r\le n/2$ we have
\begin{equation*}
  \setsize{\Lambda_n}=n^d\textq,
  \setsize{\Lambda_n^r}=(n-2r)^d\qtextq{and}
  \setsize{\partial^r(\Lambda_n)}=(n+2r)^d-(n-2r)^d\text.
\end{equation*}

In the following we introduce colorings of the elements of~$\Z^d$.
To this end, let $\cA\subseteq\R$ be the set of possible colors.
The sample set, which describes the set of all possible colorings of~$\Z^d$
is given by
\begin{equation*}
  \Omega:=\cA^{\Z^d}
        :=\{\omega=(\omega_z)_{z\in \Z^d}\mid\omega_z\in\cA\}
        \subseteq \R^{\Z^d}\text.
\end{equation*}
For each $z\in\Z^d$ we define the translation
\begin{equation}\label{eq:def:tau}
  \tau_z\from\Omega\to\Omega\textq,
  (\tau_z\omega)_x=\omega_{x+z}\textq,(z\in \Z^d)\text,
\end{equation}
i.\,e.~$\Z^d$ acts on~$\Omega$ via translations.
For $\Lambda\in\cF$ we set
$\Omega_\Lambda:=\cA^\Lambda
  :=\{\omega=(\omega_z)_{z\in\Lambda}\mid\omega_z\in\cA\}$
and define $\Pi_\Lambda\from\Omega\to \Omega_\Lambda$ by
\begin{equation*}
 \Pi_\Lambda(\omega)
   :=\omega_\Lambda
   :={(\omega_z)}_{z\in\Lambda}
   \qtextq{for}
   \omega=(\omega_z)_{z\in\Z^d}\in\Omega\text.
\end{equation*}
We simplify $\Pi_z:=\Pi_{\lbrace z\rbrace}$ for $z\in\Z^d$.
As usual, $\cA$ is equipped with the Borel $\sigma$-algebra~$\Borel(\cA)$
inherited from~$\R$.
Let~$\Borel(\Omega)$ be the product $\sigma$-algebra on~$\Omega$.
  Let~$\P$ be a probability measure on $(\Omega,\Borel(\Omega))$ satisfying:
\begin{Assumption}\label{defP}~
  \begin{enumerate}[(M1)]
   \item\label[propertyM]{M1} \emph{Translation invariance:}
     For each $z\in \Z^d$ we have $\P\circ\tau_z^{-1}=\P$.
   \item\label[propertyM]{M2} \emph{Existence of densities:}
     There are $\sigma$-finite measures~$\mu_z$, $z\in\Z^d$,
     on $(\cA,\Borel(\cA))$ such that for each $\Lambda\in\cF$ the measure
     $\P_\Lambda:=\P\circ\Pi_{\Lambda}^{-1}$
     is absolutely continuous with respect to
     $\mu_\Lambda:=\Tensor_{z\in\Lambda}\mu_z$ on
     $(\Omega_\Lambda,\Borel(\Omega_\Lambda))$.
     We denote the density function by
     $\rho_\Lambda:=\frac{\d\P_\Lambda}{\d\mu_\Lambda}$.
     The measure~$\P_\Lambda$
     is called a \emph{marginal measure} of~$\P$.
     It is defined on $(\Omega_\Lambda,\Borel(\Omega_\Lambda))$,
     where $\Borel(\Omega_\Lambda)$ is again  the product $\sigma$-algebra.
   \item\label[propertyM]{M3} \emph{Independence at a distance:}
     There exists $r\ge0$ such that for all $n\in\N$
     and non-empty $\Lambda_1,\dotsc,\Lambda_n\subseteq \Z^d$
     with $\min\{\fd(\Lambda_i,\Lambda_j)\mid i\neq j\}> r$
     we have $\rho_{\Lambda} =\prod_{j=1}^n\rho_{\Lambda_j}$,
     where $\Lambda=\bigcup_{j=1}^n \Lambda_j$.
  \end{enumerate}
\end{Assumption}

\begin{Remark}~
  \begin{itemize}
  \item The constant $r\ge0$ in~\labelcref{M3} can be interpreted as
    \emph{correlation length}.
    In particular, if $r=0$ this property implies
    that the colors of the vertices are independent.
  \item Conditions \labelcref{M2,M3} are trivially satisfied, if~$\P$
    is a product measure.
  \item For examples of measures $\P$ satisfying \labelcref{M1,M2,,M3}
    we refer to \cref{example:density}.
  \end{itemize}
\end{Remark}

In the following we deal with partial orderings on~$\Omega$
and~$\Omega_\Lambda$, $\Lambda\in\cF$.
We write $\omega\le\omega'$ for $\omega,\omega'\in\Omega$
if we have $\omega_z\le\omega'_z$ for all $z\in \Z^d$,
and analogously for $\Omega_\Lambda$.

We consider the Banach space
\begin{equation*}
  \B:=\{F\from\R\to\R\mid\text{$F$ right-continuous and bounded}\}\text,
\end{equation*}
which is equipped with supremum norm~$\norm{\argmt}:=\norm[\infty]{\argmt}$.

We now introduce a certain class of $\B$-valued functions.

\begin{Definition}\label{def:admissible}
A function $f \from\cF\times\Omega\to\B$ is called \emph{admissible}
if the following conditions are satisfied
\begin{enumerate}[(i)]
  \item\label[property]{transitive}%
    translation invariance:
    For $\Lambda\in\cF$, $z\in \Z^d$ and $\omega\in\Omega$ we have
    \begin{equation*}
      f(\Lambda+ z,\omega)=f(\Lambda,\tau_z\omega)\text.
    \end{equation*}
  \item\label[property]{local}%
    locality: For all $\Lambda\in\cF$ and $\omega,\omega'\in\Omega$
    satisfying $\Pi_\Lambda(\omega)=\Pi_\Lambda(\omega')$ we have
    \begin{equation*}
      f(\Lambda,\omega)=f(\Lambda,\omega')\text.
    \end{equation*}
  \item\label[property]{extensive}%
    almost additivity: There exists a function $b=b_f\from\cF\to\Ico0\infty$
    such that for arbitrary $\omega\in\Omega$,
    pairwise disjoint $\Lambda_1,\dotsc,\Lambda_n\in\cF$ and
    $\Lambda:=\bigcup_{i=1}^n\Lambda_i$ we have
    \begin{equation*}
      \Bignorm{f(\Lambda,\omega)-\sum_{i=1}^nf(\Lambda_i,\omega)}
        \le\sum_{i=1}^nb(\Lambda_i)\text,
    \end{equation*}
    and~$b$ satisfies
    \begin{itemize}
    \item $b(\Lambda)=b(\Lambda +z )$
      for arbitrary $\Lambda\in\cF$ and $z\in \Z^d$,
    \item $\exists D=D_f>0$ with $b(\Lambda)\le D\setsize\Lambda$
      for arbitrary $\Lambda\in\cF$,
    \item $\lim_{i\to\infty}b(\Lambda_i)/\setsize{\Lambda_i}=0$,
      if $(\Lambda_i)_{i\in\N}$ is a sequence of cubes
       with strictly increasing side length.
    \end{itemize}
  \item\label[property]{monotone}%
    coordinatewise monotonicity:
    There exists a sign vector $s\in\{-1,1\}^{\Z^d}$
    such that for all $\Lambda\in\cF$
    and all $\omega,\omega'\in\Omega$, $z\in\Lambda$ and $E\in\R$ we have
    \begin{equation*}
      \omega_{\Lambda\setminus\{z\}}=\omega'_{\Lambda\setminus\{z\}},
      \omega_z\le\omega'_z\quad\implies\quad
      \begin{cases}
        f(\Lambda,\omega)(E)\le f(\Lambda,\omega)(E)
          &\text{if $s(z)=1$, whereas}\\
        f(\Lambda,\omega)(E)\ge f(\Lambda,\omega)(E)
          &\text{if $s(z)=-1$.}
      \end{cases}
    \end{equation*}
  \item\label[property]{bounded}
    boundedness: We have
    \begin{equation*}
      \sup_{\omega\in\Omega}\norm{f(\{0\},\omega)}
        <\infty\text.
    \end{equation*}
\end{enumerate}
\end{Definition}

\begin{Remark}\label{rem:admissible}
\begin{asparaitem}
\item \Cref{local} can be formulated as follows:
  $f(\Lambda,\argmt)\colon \Omega \to \B$ is $\Pi_\Lambda$-measurable.
  \Cref{local} also enables us to define $f_\Lambda\from\Omega_\Lambda\to\B$
  by $f_\Lambda(\omega_\Lambda):=f(\Lambda,\omega)$ with $\Lambda\in\cF$
  and $\omega\in\Omega$.
  If $\setsize\Lambda=1$, we can identify~$\Omega_\Lambda=\cA^\Lambda$
  and~$\cA$.
  With this notation,~\labelcref{bounded} above translates into
  \begin{equation*}
    \sup_{a\in\cA}\norm{f_{\lbrace0\rbrace}(a)}<\infty\text.
  \end{equation*}
\item In our examples in \cref{sec:Ecf,sec:ccf},~$b(\Lambda)$
  from~\labelcref{extensive} can be chosen proportional
  to~$\setsize{\partial^1\Lambda}$,
  the size of the $1$-boundary of~$\Lambda\in\cF$.
  Accordingly, we call the function~$b=b_f$ \emph{boundary term} for~$f$.
  For quantitative estimates it is handy to require additionally
  that there exists $r'=r'_f\in\N$ and $D'=D'_f>0$ such that
  \begin{equation*}
    b(\Lambda)\le D'\setsize{\partial^{r'}\Lambda}
  \end{equation*}
  for all cubes~$\Lambda\in\cF$.
  We call such a function~$b$ a \emph{proper boundary term}.
\item
  It is natural to call $f$ with the property
  \begin{equation*}
    f(\Lambda,\omega)
      =\sum_{i=1}^nf(\Lambda_i,\omega)
  \end{equation*}
  \emph{additive} with respect to the disjoint decomposition
  $(\Lambda_i)_{i=1,\dotsc,n}$ of $\Lambda\in\cF$.
  Hence, it is again natural to call~\labelcref{extensive} almost additive,
  since the error term $\sum_{i=1}^nb(\Lambda_i)$ is in some sense small.
  Alternatively,~\labelcref{extensive} could be called \emph{low complexity}
  or \emph{semi-locality} of~$f$.
  The information contained in $f(\Lambda_1), \dotsc, f(\Lambda_n)$
  does not differ much from the information contained in~$f(\Lambda)$.
\item
  Our examples in \cref{sec:Ecf,sec:ccf}
  deal with antitone admissible functions,
  i.\,e.~\labelcref{monotone} is satisfied with $s(z)=-1$ for all $z\in\Z^d$.
\item If~$f$ is admissible, then
  \begin{align}\label{def:Cf}
    \K_f:=\sup\Bigl\{\tfrac{\norm{f(\Lambda,\omega)}}{\setsize{\Lambda}}
             \Bigm|\omega\in\Omega,\Lambda\in\cF\setminus\{\emptyset\}
             \Bigr\}
    <\infty\text.
  \end{align}
  To see this, we choose $\Lambda\in\cF$ and $\omega\in\Omega$
  arbitrarily and calculate as follows:
  \begin{align*}
   \norm{f(\Lambda,\omega)}
  &\le\Bignorm{f(\Lambda,\omega)-\sum_{z\in\Lambda}
     f(\{z\},\omega)}+\Bignorm{\sum_{z\in\Lambda}f(\{z\},\omega)}\\
  &\le\sum_{z\in\Lambda}b(\{z\})+\sum_{z\in\Lambda}\norm{f(\{z\},\omega)}\\
  &\le D\setsize{\Lambda}+\sum_{z\in\Lambda}\norm{f(\{0\},\tau_{-z}\omega)}
   \le\bigl(D+\sup_{\omega\in\Omega}\norm{f(\{0\},\omega)}\bigr)
       \setsize{\Lambda}\text.
  \end{align*}
  Thus, $\K_f\le D+\sup_{\omega\in\Omega}\norm{f(\{0\},\omega)}<\infty$.
\end{asparaitem}
\end{Remark}

\begin{Definition}
  For $\K,D,D'>0$ and $r'\in\N$,
  we form the set
  \begin{equation*}
    \cU_{\K,D,D',r'}
    =
    \{f:\cF \times \Omega\to\B \mid f\text{ admissible with }
      \K_f\le \K, D_f\le D, D'_f\le D'\text{ and } r'_f\le r'\}
  \end{equation*}
  where $K_f$, $D_f$, $D'_f $ and $r'_f$
  are the constants from \cref{def:admissible} associated to~$f$.
\end{Definition}

Let us state the main theorem of this paper.

\begin{Theorem}\label{thm:main}
  Let~$\cA\subseteq\R$, $\Omega:=\cA^{\Z^d}$ and $(\Omega,\Borel(\Omega),\P)$
  a probability space such that~$\P$ satisfies \labelcref{M1,M2,,M3}
  with correlation length $r\in\N_0$,
  and let $f\colon\cF\times\Omega\to\B$ be an admissible function.
  Let further $\Lambda_n:=\Ico0n\isect\Z^d$ for $n\in\N$.
  Then there exists a set $\tilde\Omega\in\Borel(\Omega)$
  of full measure and a function $f^*\in\B$
  such that for every $\omega\in\tilde\Omega$ we have
  \begin{equation}\label{eqmain}
    \lim_{n\to\infty}\biggnorm{
      \frac{f(\Lambda_n,\omega)}{\setsize{\Lambda_n}}-f^*}
      =0\text.
  \end{equation}
\end{Theorem}

\begin{Remark}\label{rem:measurable}
\begin{asparaitem}
\item The following special case illustrates the relation to the
  Glivenko--Cantelli theorem.
  Let $\P:=\bigotimes_{z\in\Z}\mu$ be a product measure on $\bigtimes_{\ZZ} \RR$,
  where~$\mu$ is a probability measure on~$\R$,
  and let $f(\Lambda,\omega)(E):=\sum_{z\in\Lambda}\ifu{\Ioc{-\infty}E}(\omega_z)$
  for $\Lambda\in\cF$, $\omega\in\Omega$ and $E\in\R$.
  Then~$f$ is an admissible function.
  The quantities
  $f(\Lambda_n,\omega)(E)/\setsize{\Lambda_n}
    =\setsize{\Lambda_n}^{-1}\sum_{z\in\Lambda_n}
      \dirac{\omega_z}(\Ioc{-\infty}E)$
  turn out to be the distribution functions of empirical measures.
  \cref{thm:main}
  now states that the empirical distribution functions converge uniformly.
  The limit~$f^*$ is of course the distribution function of~$\mu$:
  $f^*(E)=\mu(\Ioc{-\infty}E)$ for all $E\in\R$.
\item We emphasize that the statement of \cref{thm:main}
  does not contain the measurability of the set
  \begin{equation*}
    \Bigl\{\omega\in\Omega\Bigm|\lim_{n\to\infty}
      \Bignorm{\tfrac{f(\Lambda_n,\omega)}{\setsize{\Lambda_n}}-f^*}=0
    \Bigr\}\text.
  \end{equation*}
  Instead, the claim is that this set contains a measurable subset~%
  $\tilde\Omega$ of full measure.
  If the probability space was complete, the above set would be measurable, too.
  We write all almost sure statements in explicit fashion,
  in order to avoid a completeness assumption and measurability issues.
\item The limit function~$f^*$ inherits the boundedness from~$f$,
  since there exists $\omega\in\Omega$ such that
  \begin{equation*}
    \norm{f^*}
      \le\limsup_{n\to\infty}\Bigl(
        \Bignorm{\tfrac{f(\Lambda_n,\omega)}{\setsize{\Lambda_n}}-f^*}
        +\Bignorm{\tfrac{f(\Lambda_n,\omega)}{\setsize{\Lambda_n}}}\Bigr)
      \le\K_f\text.
  \end{equation*}
\item Note that \cref{thm:main} readily generalizes to absolutely convergent
  linear combinations of admissible functions in the following sense.
  Let $\K,\alpha_j\in\R$, $j\in\N$ such that
  $\sum_{j\in\N}\abs{\alpha_j}<\infty$ and~$f_j$, $j\in\N$,
  admissible functions such that $\K_{f_j}\le \K$ for all $j\in\N$.
  For each $j\in\N$, \cref{thm:main} provides a limit function~$f_j^*$.
  We let $f:=\sum\alpha_jf_j$ and $f^*:=\sum\alpha_jf_j^*$.
  Now fix $\ve>0$, find $J\in\N$ such that
  $\sum_{j=J}^\infty\abs{\alpha_j}<\frac\ve{2\K}$
  and note that by the triangle inequality
  for all $\omega\in\tilde\Omega$ we have
  \begin{align*}
    \limsup_{n\to\infty}
      \Bignorm{\tfrac{f(\Lambda_n,\omega)}{\setsize{\Lambda_n}}-f^*}
   &\le\limsup_{n\to\infty}\biggl(\sum_{j=1}^{J-1}\abs{\alpha_j}
      \Bignorm{\tfrac{f_j(\Lambda_n,\omega)}{\setsize{\Lambda_n}}-f_j^*}
      +\sum_{j=J}^\infty\abs{\alpha_j}\Bigl(
        \Bignorm{\tfrac{f_j(\Lambda_n,\omega)}{\setsize{\Lambda_n}}}
          +\norm{f_j^*}\Bigr)\biggr)\\
   &<\ve\text.
  \end{align*}
  This shows that the coordinate-wise monotonicity~\labelcref{monotone} can be somewhat weakened.
  \item By a Borel--Cantelli argument employing \cref{wright:LDP},
    the sequence of cubes~$\Lambda_n$
    can in fact be replaced by an arbitrary van Hove sequence~%
    $(\Lambda_n)_{n\in\N}$, as long as for each~$\Lambda_n$
    there exists a collection of translates which tiles~$\Z^d$.
    The set~$\tilde\Omega$ will depend on the sequence, of course.
\end{asparaitem}
\end{Remark}


Next, we state that for functions with proper boundary terms
the convergence in \cref{thm:main} can be quantified by error terms.
For the definition of the empirical measure $L_{m,n}^{r,\omega}$
see~\eqref{def:empmeasure:new} and the notation $\spr{f}{\nu}$ is introduced in~\eqref{def:sprfnu}.
\begin{Theorem}\label{cor:main}
  Let~$\cA\subseteq\R$, $\Omega:=\cA^{\Z^d}$ and $(\Omega,\Borel(\Omega),\P)$
  a probability space such that~$\P$ satisfies \labelcref{M1,M2,,M3}
  with correlation length $r\in\N_0$,
%
%
  and let $\Lambda_n:=\Ico0n\isect\Z^d$ for $n\in\N$.
  Let $\K,D,D'>0$ and $r'\in\N$.
  Then, there exists a set $\tilde\Omega\in\Borel(\Omega)$
  of full probability such that, for each $m,n\in\N$ with $n>2m>4r$
  and $\omega\in\tilde\Omega$, we have
  \begin{multline*}
    \sup_{f\in\cU_{\K,D,D',r'}}
      \Bignorm{\frac{f(\Lambda_n,\omega)}{\setsize{\Lambda_n}}- f^*}
      \le2^{2d+1}\Bigl(\frac{(2\K+D)m^d+D'{r'}^d}{n-2m}
        +\frac{2(\K+D)r^d+3D'{r'}^d}{m-2r}\Bigr)\\
        +\sup_{f\in\cU_{\K,D,D',r'}}\frac{\bignorm{\spr{f_{\Lambda_m^r}}
        {L_{m,n}^{r,\omega}-\P_{\Lambda_m^r}}}}{\setsize{\Lambda_m}}\text,
  \end{multline*}
  where $f^*$ is the limit given by \cref{thm:main} applied to~$f$.
  Furthermore, for all $m\in\N$ and $\omega\in\tilde\Omega$,
  \begin{equation*}
    \lim_{n\to\infty}\sup_{f\in\cU_{\K,D,D',r'}}
      \bignorm{\spr{f_{\Lambda_m^r}}{L_{m,n}^{r,\omega}-\P_{\Lambda_m^r}}}
      =0\text.
  \end{equation*}
  Even more, for each $\epsilon >0$ there exist $a=a(\epsilon,m,\K)>0$
  and $b=b(\epsilon,m,\K)$ such that for all $n\in\N$
  there is a measurable set $\Omega(\epsilon,n)$ with
  $\P(\Omega(\epsilon,n))\ge1-b\exp(-a\floor{n/m}^d)$ and
  \begin{equation*}
    \sup_{f\in\cU_{\K,D,D',r'}}
      \bignorm{\spr{f_{\Lambda_m^r}}{L_{m,n}^{r,\omega}-\P_{\Lambda_m^r}}}
      \le\epsilon
      \qtext{for all $\omega \in \Omega(\epsilon,n)$.}
  \end{equation*}
\end{Theorem}

\begin{Remark}
  \begin{asparaitem}
  \item It would be interesting to find an optimal explicit expression for~$a$
    and~$b$ in terms of $\epsilon$, $m$, and~$\K$.
  \item As before, the monotonicity can be weakened, see \cref{rem:measurable}.
    Note in particular, that a convex combination of functions in
    $\cU_{\K,D,D',r'}$ still obeys the quantitative estimates given by
    $\K,D,D'$ and $r'$.
    The statement of \cref{cor:main} remains valid for the convex combination
    since the geometric part is derived without the use of monotonicity
    and the argument from \cref{rem:measurable}
    applies to the probabilistic part.
  \end{asparaitem}
\end{Remark}

\begin{Corollary}
  Under the conditions of \cref{cor:main},
  there exists a set $\tilde\Omega\in\cB(\Omega)$ with $\P(\tilde\Omega)=1$
  such that for all $\omega\in\tilde\Omega$, we have
  \begin{equation}\label{eq:reformulated}
    \sup_{f\in\cU_{K,D,D',r'}}\sup_{E\in\R}
      \biggabs{\frac{f(\Lambda_n,\omega)(E)}{\setsize{\Lambda_n}}-f^*(E)}
      \xto{n\to\infty}0\text.
  \end{equation}
  If furthermore for an admissible~$f$, $f(\Lambda_n,\omega)\from\R\to\R$
  is an isotone function for all~$\Lambda_n$ and $\omega\in\tilde\Omega$,
  then the limit function~$f^*\in\B$ is isotone, too.
  In particular, cummulative distribution functions are preserved.
\end{Corollary}
\begin{proof}
  By \cref{cor:main}, we have
  \begin{equation*}
    0\le
    \lim_{n\to\infty}\sup_{f\in\cU_{K,D,D',r'}}
      \Bignorm{\frac{f(\Lambda_n,\omega)}{\setsize{\Lambda_n}}-f^*}
    \le2^{2d+1}\frac{2(K+D)r^d+3D{r'}^d}{m-2r}
    \xto{m\to\infty}0\text.
  \end{equation*}
  Recall that the norm in $\B$ is the sup norm
  $\norm\argmt=\sup_{E\in\R}\abs{\argmt(E)}$ to see \eqref{eq:reformulated}.
  \par
  If the functions
  $f_{n,\omega}:=f(\Lambda_n,\omega)/\setsize{\Lambda_n}\from\R\to\R$
  are increasing, then for all $E,E'\in\R$ with $E<E'$ and $\varepsilon>0$
  we find $n\in\N$ such that $\norm{f_{n,\omega}-f^*}<\ve/2$ and
  \begin{equation*}
    f^*(E)
    \le f_n(E)+\ve/2
    \le f_n(E')+\ve/2
    \le f^*(E')+\ve\text.
  \end{equation*}
  Since $\ve>0$ was arbitrary, $f^*$ is increasing, too.
\end{proof}

\section{Illustration of the idea of proof}\label{sec:wunsch}

Let us consider the exemplary situation of dimension $d=1$
and independently chosen colors, i.\,e.,
the constant~$r$ from~\labelcref{M3} equals~$0$.
In this case, the idea of the proof of \cref{thm:main}
is illustrated in the following line:
\begin{align}\label{wunsch}
  \frac1{mk}f\bigl(\Ico0{mk},\omega\bigr)&
    \stackrel{(1)}{\approx}\frac1 m\spr{f_m}{L_{m,mk}^\omega}
    \stackrel{(2)}{\approx}\frac1 m\spr{f_m}{\P_m}
    \stackrel{(3)}{\approx} f^*\text,
\end{align}
where $0\ll m\ll k$.
Assume that $n=mk$ and $\Lambda_n=\Ico0n$.
Then the left hand side in~\eqref{wunsch}
equals the approximant in \cref{thm:main}.
The function $f_m\from\Omega_{\Ico0m}\to\B$ is defined by
$f_m(\omega):=f(\Ico0m,\omega')$
for $\omega'\in\Pi^{-1}_{\Ico0m}(\{\omega\})$,
cf.\ \cref{rem:admissible}.
$L_{m,n}^\omega(B)$ is the empirical probability measure
counting the number of occurrences of elements of $B\in\Borel(\Omega_{\Ico0m})$
at the positions $\Ico{jm}{(j+1)m}$, $j=0,1,\dotsc,k-1$ in~$\omega$,
i.\,e.
\begin{equation}\label{eq:def:empmeasure0}
  L_{m,n}^\omega \from \Borel(\Omega_{\Ico0m})\to\Icc01\textq,
  L_{m,n}^\omega
   :=\frac1k\sum_{j=0}^{k-1}\dirac{(\tau_{jm}\omega)_{\Ico0m}}
    =\frac1k\sum_{j=0}^{k-1}\dirac{\Pi_{\Ico0m} (\tau_{jm} \omega)}\text.
\end{equation}
We use the shortcut notation
\begin{equation*}
  \spr{f_m}{L_{m,n}^\omega}
    :=\int_{\Omega_{\Ico0m}}f_m(\omega')\,\d L_{m,n}^\omega(\omega')
    =\frac1k\sum_{j=0}^{k-1}f\bigl(\Ico{jm}{(j+1)m},\omega\bigr)
    \text.
\end{equation*}
Let us discuss the three approximation steps separately.
\begin{enumerate}[(1)]
\item
  The ``$\stackrel{(1)}\approx$''
  means that the two expressions are getting close to each other if~$m$
  increases.
  To see this we use almost additivity of an admissible function,
  cf.~\labelcref{extensive} of \cref{def:admissible}.
  The detailed calculations will be presented in \cref{secboundary}.
\item In the second step we compare the empirical measure $L_{m,mk}^\omega$
  with the marginal measure $\P_{m}:=\P_{\Ico0m}$.
  The method of choice is a multivariate Glivenko--Cantelli theorem,
  which we apply in a version of DeHardt and Wright.
  In this particular situation it shows that for increasing~$k$
  the expression $\spr{f_m}{L_{m,mk}^\omega}$
  converges to $\spr{f_m}{\P_m}$ almost surely.
  This approximation step is explicitly discussed in \cref{GC}.
\item In the last step we make intensive use of almost additivity of~$f$
  in order to obtain that $(\spr{f_m}{\P_m}/m)_m$ is Cauchy sequence in~$\B$.
  As~$\B$ is a Banach space, we can identify the limit with an element
  $f^*\in\B$.
  The details are found in \cref{sec:Aadd}.
\end{enumerate}

\begin{Remark}[Frequencies]
  From the discussion of step~(1)
  above it is clear that the empirical measure
  counts occurrences of patterns at the
  positions $\Ico{jm}{(j+1)m}$ for $j=0,1,\dotsc,k-1$.
  Thus, the corresponding sets are disjoint
  and their union covers the whole interval $\Ico0n$, $n=mk$.
  In this sense, the present technique of counting occurrences
  differs from the counting in certain papers.
  For instance in \cite{LenzMV-08,LenzSV-10},
  the authors count occurrences of patterns
  at each possible consecutive position,
  ignoring the fact that they may overlap.
  In our setting, this would correspond to the situation
  where the empirical measure is defined
  to count occurrences at all positions
  $\Ico j{j+m}$, $j=0,1,\dotsc,m(k-1)$, i.\,e.,
  \begin{align}\label{eq:def:barL}
    \bar L_{m,n}^{\omega}\from\Borel(\Omega_{\Ico0m})\to\Icc01\textq,
    \bar L_{m,n}^\omega
      :=\frac1{m(k-1)+1}\sum_{j=0}^{m(k-1)}\dirac{(\tau_j\omega)_{\Ico0m}}
      \text.
  \end{align}
  However, both ways of counting can be related to each other.
  The link can be seen best by comparing with the average
  \begin{align}\label{eq:barL1}
    \frac1{m(k-1)}\sum_{j=1}^{m(k-1)}\dirac{(\tau_j\omega)_{\Ico0m}}
    =\frac1m\sum_{i=1}^{m}\;
      \frac1{k-1}\sum_{j=0}^{k-2}\dirac{(\tau_{jm+i}\omega)_{\Ico0m}}
      \text,
  \end{align}
  where the first observation~$\delta_{\omega_{\Ico0m}}$ is discarded.
  Indeed, for large $n=mk$, the difference between~$\bar L_{m,n}^\omega$
  and~\eqref{eq:barL1} vanishes.
  The right hand side of~\eqref{eq:barL1} shows that $\bar L_{m,n}^\omega$
  is essentially a convex combination of empirical measures of the type~%
  \eqref{eq:def:empmeasure0}.
  As $k\to\infty$, all the empirical measures of type~\eqref{eq:def:empmeasure0}
  in~\eqref{eq:barL1} converge to the same limit~$\P_m$,
  rendering the convex combination harmless.
  Recall that in the approximation first the limit $k\to\infty$
  and afterwards the limit $m\to\infty$ is performed.
  This shows that the empirical measure defined in~\eqref{eq:def:barL}
  converges to the same limit as the empirical measures in~%
  \eqref{eq:def:empmeasure0}.
\end{Remark}

\section{Approximation via the empirical measure}\label{secboundary}

In the following we show how to estimate an admissible function~$f$
in terms of the empirical measure.
As in \cref{thm:main}, let $\Lambda_n=(\Ico0n\isect\Z)^d$ for each $n\in\N$.

Our aim is to approximate for $m\ll n$ the set~$\Lambda_n$
using translates of the set~$\Lambda_m$.
To this end, we define the grid
\begin{equation}\label{def:Tmn}
  T_{m,n}:=
    \{t\in m\Z^d\mid\Lambda_m+t\subseteq \Lambda_n\}
    \text.
\end{equation}
Thus, we have $\setsize{T_{m,n}}=\floor{n/m}^d$,
 $\Lambda_{\floor{n/m}m}
    =\dUnion_{t\in T_{m,n}}(\Lambda_m+t)
    =\Lambda_m+T_{m,n}$, and
\begin{equation}\label{eq:tilingconstruction}
    \Lambda_n\setminus\Lambda_{\floor{n/m}m}
      \subseteq\partial^m(\Lambda_n)
    \qtextq{or equivalently}
    \Lambda_n^m\subseteq\Lambda_{\floor{n/m}m}\text.
\end{equation}



As in \cref{rem:admissible}, define for an admissible~$f$
and $\Lambda\in\cF$ the function
\begin{equation}\label{eq:def:fLambda}
  f_\Lambda\from\Omega_{\Lambda}\to\B, \qquad f_\Lambda(\omega)
    :=f(\Lambda,\omega')\qqtext{where
      $\omega'\in\Pi_{\Lambda}^{-1}(\{\omega \})$.}
\end{equation}
By locality~\labelcref{local} of \cref{def:admissible},
$f_\Lambda$ is well-defined.
In the case~$\Lambda=\Lambda_n$, we write
\begin{equation}\label{eq:def:f_n}
  f_n:=f_{\Lambda_n} \qquad\text{and}\qquad f_n^m:=f_{\Lambda_n^m}
\end{equation}
for $m\in\N_0$.
Next, we introduce the empirical measure $L_{m,n}^\omega$ by setting
for $\omega\in\Omega$ and $m,n\in\N$:
\begin{equation}\label{def:Lmn}
  L_{m,n}^\omega \from \Borel(\Omega_{\Lambda_m})\to\Icc01\textq,
  L_{m,n}^\omega
    =\frac1{\setsize{T_{m,n}}}\sum_{t\in T_{m,n}}
      \dirac{(\tau_t \omega)_{\Lambda_m}}\text.
\end{equation}
Here, $\dirac\omega\from\Borel(\Omega_{\Lambda_m})\to\Icc01$
is the point measure on $\omega\in\Omega_{\Lambda_m}$.
In the same manner, we define $L_{m,n}^{r,\omega}$
as an adaption of $L_{m,n}^{\omega}$ which ignores the $r$-boundary of~$\Lambda_m$.
The precise definition is the following:
for $r\in\N_0$ we set
\begin{equation}\label{def:empmeasure:new}
  L_{m,n}^{r,\omega} \from \Borel(\Omega_{\Lambda_m^r})\to\Icc01\textq,
  L_{m,n}^{r,\omega}
    =\frac1{\setsize{T_{m,n}}}\sum_{t\in T_{m,n}}
       \dirac{(\tau_t \omega)_{\Lambda_m^r}}\text.
\end{equation}
The variable~$r$
we used here will eventually be the constant from~\labelcref{M3},
but here in \cref{secboundary} we do not need that specific value.

As illustrated before in~\cref{sec:wunsch} we use for $\Lambda\in\cF$,
a bounded and measurable $f\from\Omega_\Lambda\to\B$,
and a measure~$\nu$ on $(\Omega_\Lambda,\Borel(\Omega_\Lambda))$ the notation
\begin{equation}\label{def:sprfnu}
  \spr f\nu:=\int_{\Omega_{\Lambda}}f(\omega)\,\d\nu(\omega)\text.
\end{equation}

\begin{Lemma}\label{la:empmeasure}
  Recall $\Lambda_n:=(\Ico0n\isect\Z)^d$.
  For any admissible function $f\from\cF\times\Omega\to\B$
  we have, for all $\omega\in\Omega$
  and all $n,m,r\in\N$ with $n>2m$,
  \begin{multline}\label{eq:est1:empmeasure}
    \biggnorm{\frac{f(\Lambda_n,\omega)}{n^d}
      -\frac{\spr{f_m^r}{L_{m,n}^{r,\omega}}}{m^d}}
      \le\frac{b(\Lambda_{\floor{n/m}m})}
              {\setsize{\Lambda_{\floor{n/m}m}}}
        +\frac{(2\K_f+D)\setsize{\partial^m(\Lambda_{n}^m)}}
              {\setsize{\Lambda_n^m}}+{}\\
        +\frac{b(\Lambda_m)+b(\Lambda_m^r)
                +(\K_f+D)\setsize{\partial^r(\Lambda_m)}}
              {\setsize{\Lambda_m}}\text.
  \end{multline}
  Moreover,
  \begin{equation}\label{eq:limit:empmeasure}
    \lim_{m\to\infty}\lim_{n\to\infty}
      \biggnorm{\frac{f(\Lambda_n,\omega)}{n^d}
               -\frac{\spr{f_m^r}{L_{m,n}^{r,\omega}}}{m^d}}
    =0\text.
  \end{equation}
\end{Lemma}
\begin{proof}
  Let $\omega\in\Omega$ and $n,m,r\in\N$ be given such that $n>2m$.
  This condition ensures that $\Lambda_n^m\ne\emptyset$.
  First we verify~\eqref{eq:est1:empmeasure},
  and afterwards we show that this implies~\eqref{eq:limit:empmeasure}.
  By the triangle inequality we obtain
  \begin{multline}\label{eq:triangle}
    \biggnorm{\frac{f(\Lambda_n,\omega)}{n^d}
             -\frac{\spr{f_m^r}{L_{m,n}^{r,\omega}}}{m^d}}
    \le\biggnorm{\frac{f(\Lambda_n,\omega)}{\setsize{\Lambda_n}}
                -\frac{f(\Lambda_n,\omega)}
                      {\setsize{\Lambda_{\floor{n/m}m}}}}
      +\biggnorm{\frac{f(\Lambda_n,\omega)-f(\Lambda_{\floor{n/m}m},\omega)}
                      {\setsize{\Lambda_{\floor{n/m}m}}}}\\
      +\biggnorm{\frac{f(\Lambda_{\floor{n/m}m},\omega)}
                      {\setsize{\Lambda_{\floor{n/m}m}}}
                -\frac{\spr{f_m}{L_{m,n}^\omega}}{m^d}}
      +\frac{\norm{\spr{f_m}{L_{m,n}^\omega}-\spr{f_m^r}{L_{m,n}^{r,\omega}}}}
            {m^d}\text.
  \end{multline}
  We bound the four terms on the right hand side separately.
  To estimate the first term,
  we use $\setsize{\Lambda_{\floor{n/m}m}}\ge\setsize{\Lambda_n^m}$,
  see~\eqref{eq:tilingconstruction}, and obtain
  \begin{equation*}
    0\le\frac1{\setsize{\Lambda_{\floor{n/m}m}}}-\frac1{\setsize{\Lambda_n}}
    \le\frac1{\setsize{\Lambda_n^m}}-\frac1{\setsize{\Lambda_n}}
    =\frac{\setsize{\Lambda_n}-\setsize{\Lambda_n^m}}
          {\setsize{\Lambda_n}\setsize{\Lambda_n^m}}
    \le\frac{\setsize{\partial^m(\Lambda_n^m)}}
            {\setsize{\Lambda_n}\setsize{\Lambda_n^m}}
    \text.
  \end{equation*}
  Applying the bound given by~\eqref{def:Cf} in \cref{rem:admissible}, we get
  \begin{equation}\label{eq:bdest1}
    \biggnorm{\frac{f(\Lambda_n,\omega)}{\setsize{\Lambda_n}}
             -\frac{f(\Lambda_n,\omega)}
                   {\setsize{\Lambda_{\floor{n/m}m}}}}
    \le \K_f\frac{\setsize{\partial^m(\Lambda_n^m)}}{\setsize{\Lambda_n^m}}
    \text.
  \end{equation}
  In order to find an appropriate upper bound for the second term in~%
  \eqref{eq:triangle}
  we use almost additivity~\labelcref{extensive},
  the inclusion~\eqref{eq:tilingconstruction}
  and $\hat\Lambda_{m,n}:=\Lambda_n\setminus\Lambda_{\floor{n/m}m}$ to obtain
  \begin{align}
    \biggnorm{\frac{f(\Lambda_n,\omega)-f(\Lambda_{\floor{n/m}m},\omega)}
             {\setsize{\Lambda_{\floor{n/m}m}}}}
    &\le\frac{b(\Lambda_{\floor{n/m}m})}{\setsize{\Lambda_{\floor{n/m}m}}}
       +\frac{b(\hat \Lambda_{n,m})}{\setsize{\Lambda_{\floor{n/m}m}}}
       +\frac{\norm{f(\hat\Lambda_{n,m},\omega)}}
             {\setsize{\Lambda_{\floor{n/m}m}}}\nonumber\\
    &\le\frac{b(\Lambda_{\floor{n/m}m})}{\setsize{\Lambda_{\floor{n/m}m}}}
       +\frac{D\setsize{\hat\Lambda_{n,m}}}{\setsize{\Lambda_{\floor{n/m}m}}}
       +\frac{\K_f\setsize{\hat\Lambda_{n,m}}}{\setsize{\Lambda_{\floor{n/m}m}}}
         \nonumber\\
    &\le\frac{b(\Lambda_{\floor{n/m}m})}{\setsize{\Lambda_{\floor{n/m}m}}}
       +\frac{(\K_f+D)\setsize{\partial^m(\Lambda_{n}^m)}}{\setsize{\Lambda_n^m}}
    \text.
    \label{eq:bdest2}
  \end{align}
  To approximate the third term in~\eqref{eq:triangle},
  we calculate using translation invariance~\labelcref{transitive}
  of admissible functions
  \begin{align}
    \spr{f_m}{L_{m,n}^\omega}
    &=\int_{\Omega_{\Lambda_m}}f_m(\omega')\,\d L_{m,n}^\omega(\omega')
    =\frac1{\setsize{T_{m,n}}}\sum_{t\in T_{m,n}}\int_{\Omega_{\Lambda_m}}
      f_m(\omega')\,\d\dirac{(\tau_t\omega)_{\Lambda_m}}(\omega')\notag\\
    &=\frac1{\setsize{T_{m,n}}}\sum_{t\in T_{m,n}}
      f_m((\tau_t\omega)_{\Lambda_m})
    =\frac1{\setsize{T_{m,n}}}\sum_{t\in T_{m,n}}f(\Lambda_m +t,\omega)
    \text.\label{eq:Lmn}
  \end{align}
  This and~\labelcref{extensive} of \cref{def:admissible} give
  \begin{align}
    \biggnorm{\frac{f(\Lambda_{\floor{n/m}m},\omega)}
                   {\setsize{\Lambda_{\floor{n/m}m}}}
             -\frac{\spr{f_m}{L_{m,n}^\omega}}{\setsize{\Lambda_m}}}
    &=\frac1{\setsize{T_{m,n}}\setsize{\Lambda_m}}
      \biggnorm{f(\Lambda_{\floor{n/m}m},\omega)
               -\sum_{t\in T_{m,n}}f(\Lambda_m+t,\omega)}\notag\\
    &\le\frac1{\setsize{T_{m,n}}\setsize{\Lambda_m}}
      \sum_{t\in T_{m,n}}b(\Lambda_m+t)
    =\frac{b(\Lambda_m)}{\setsize{\Lambda_m}}
    \text.\label{eq:bdest3}
  \end{align}
  Finally, we estimate the fourth term.
  In the same way as in~\eqref{eq:Lmn} we obtain
  \begin{equation*}
    \spr{f_m^r}{L_{m,n}^{r,\omega}}
      =\frac1{\setsize{T_{m,n}}}\sum_{t\in T_{m,n}}f(\Lambda_m^r+t,\omega)
      \text.
  \end{equation*}
  Application of the triangle inequality,
  $\Lambda_m\setminus\Lambda_m^r=\Lambda_m\isect\partial^r(\Lambda_m)
    \subseteq\partial^r(\Lambda_m)$
  and~\labelcref{extensive} of \cref{def:admissible}
  as well as~\eqref{def:Cf} lead to
  \begin{align}
    \norm{\spr{f_m}{L_{m,n}^{\omega}}-\spr{f_m^r}{L_{m,n}^{r,\omega}}}
    &\le\frac1{\setsize{T_{m,n}}}\sum_{t\in T_{m,n}}
      \norm{f(\Lambda_m+t,\omega)-f(\Lambda_m^r+t,\omega)}\notag\\
    &\le\frac1{\setsize{T_{m,n}}}\sum_{t\in T_{m,n}}
      \bigl(b(\Lambda_m^r)+b(\Lambda_m\setminus\Lambda_m^r)
            +\norm{f((\Lambda_m\setminus\Lambda_m^r)+t,\omega)}\bigr)\notag\\
    &\le b(\Lambda_m^r)+(\K_f+D)\setsize{\partial^r(\Lambda_m)}
      \text.\label{eq:bdest4}
  \end{align}
  It remains to combine~\eqref{eq:triangle}
  with the bounds~\eqref{eq:bdest1}, \eqref{eq:bdest2}, \eqref{eq:bdest3}
  and~\eqref{eq:bdest4} to obtain~\eqref{eq:est1:empmeasure}.
  \par
  Let us turn to~\eqref{eq:limit:empmeasure}.
  As required, we first perform the limit $n\to\infty$.
  In~\eqref{eq:est1:empmeasure},
  the bounding terms affected by this limit vanish,
  due to \cref{extensive} and the fact that~$\Z^d$ is amenable:
  \begin{equation*}
    \lim_{n\to\infty}\biggl(\frac{b(\Lambda_{\floor{n/m}m})}
                                {\setsize{\Lambda_{\floor{n/m}m}}}
    +\frac{(2\K_f+D)\setsize{\partial^m(\Lambda_{n}^m)}}
          {\setsize{\Lambda_n^m}}\biggr)
      =0\text.
  \end{equation*}
  Secondly, we let $m\to\infty$.
  Since
  $b(\Lambda_m^r)/\setsize{\Lambda_m}\le b(\Lambda_m^r)/\setsize{\Lambda_m^r}$
  for $m>2r$, this takes care of the remaining terms of the upper bound in~%
  \eqref{eq:est1:empmeasure}.
  \begin{equation*}
    \lim_{m\to\infty}\frac{b(\Lambda_m)+b(\Lambda_m^r)
            +(\K_f+D)\setsize{\partial^r(\Lambda_m)}}{\setsize{\Lambda_m}}
      =0\text.
  \end{equation*}
  Thus,~\eqref{eq:limit:empmeasure} follows.
\end{proof}

\begin{Remark}
  Let us emphasize that the statement of the lemma is not an
  ``almost sure''-statement, but rather holds for all $\omega\in\Omega$.
\end{Remark}

\section{Application of Multivariate Glivenko--Cantelli Theory}\label{GC}

We briefly restate multivariate Glivenko--Cantelli results in \cref{wright:LDP}
and apply this result to our setting in \cref{thm:mainGK}.
To do so, we need some notions concerning monotonicity in~$\RL$.
\begin{Definition}\label{monodef}
  Let $\di\in\N$.
  \begin{compactitem}
    \item Let $s\in\{-1,1\}^\di$.
      The closed \emph{cone} $\cone_s$ with \emph{sign vector}~$s$ is the set
      \begin{equation*}
        \cone_s:=\{x=(x_j)_{j=1,\dotsc,\di}\in\RL\mid
          \forall j\in\{1,\dotsc,\di\}\colon x_js_j\ge0\}\text.
      \end{equation*}
      The closed cone with sign vector~$s$ and \emph{apex} $x\in\RL$
       is $\cone_s(x):=x+\cone_s$.
    \item A function $f\from\RL\to\R$ is \emph{monotone},
      if it is monotone in each coordinate,
      i.\,e.\ there exists $s\in\{-1,1\}^\di$
      such that, for all $x,y\in\RL$,
      \begin{equation*}
        y\in\cone_s(x)\implies
        f(y)\ge f(x)\text.
      \end{equation*}
    \item A set $\graph\subseteq\RL$ is a \emph{monotone graph},
      if there exists a sign vector $s\in\{-1,1\}^\di$
      such that, for all $x\in\graph$,
      \begin{equation*}
        \graph\isect\cone_s(x)\subseteq\bdry\cone_s(x)\text,
      \end{equation*}
      where $\bdry C$ denotes the boundary of~$C$ in $\RL$.
    \item A set $\graph\subseteq\RL$ is a \emph{strictly monotone graph},
      if there exists a sign vector $s\in\{-1,1\}^\di$
      such that, for all $x\in\graph$,
      \begin{equation*}
        \graph\isect\cone_s(x)=\{x\}\text.
      \end{equation*}
  \end{compactitem}
\end{Definition}
\begin{Remark}~
  \begin{itemize}
  \item This notion of monotonicity is compatible with~\labelcref{monotone}
    in \cref{def:admissible}.
  \item We want to emphasize that in the above definition a second meaning
    of the notion of a \emph{graph} was used.
    In \cref{sec:notation} a graph was introduced as a pair consisting of a set
    of vertices and a set of edges.
    In contrast to that, \cref{monodef} states that a \emph{monotone graph}
    is a certain subset of~$\RL$.
    In order to distinguish both meanings we will always insert the term
    \emph{monotone} when speaking about subsets of~$\RL$.
  \end{itemize}
\end{Remark}

The following theorem is proven in \cite[Theorem~1 and~2]{Wright1981}.
Recall that the continuous part~$\muc$ of a measure~$\mu$ on~$\RL$
is given by $\muc(A):=\mu(A)-\sum_{x\in A}\mu\{x\}$
for all Borel sets $A\in\Borel(\RL)$.
\begin{Theorem}[DeHardt, Wright]\label{wright:LDP}
  Let $(\Omega,\cA,\P)$ be a probability space and $X_t\from\Omega\to\RL$,
  $t\in\N$, independent and identically distributed random variables
  with distribution~$\mu$, i.\,e., $\mu(A):=\P(X_1\in A)$
  for all $A\in\Borel(\RL)$.
  For each $J\subseteq\{1,\dotsc,\di\}$, $J\ne\emptyset$,
  let $\mu^J$ be the distribution of the vector~$(X_1^j)_{j\in J}$
  consisting of the coordinates $j\in J$
  of the vector~$X_1=(X_1^j)_{j\in\{1,\dotsc,\di\}}$,
  i.\,e.\ a marginal of~$\mu$.
  We denote by
  \begin{equation*}
    L_n\from\Omega\times\Borel(\RL)\to\R\textq,
    L_n^{(\omega)}(A):=\frac1n\sum_{t=1}^n\dirac{X_t(\omega)}
  \end{equation*}
  the empirical distribution corresponding to the sample
  $(X_1,\dotsc,X_n)$, $n\in\N$.
  Fix further $\M>0$ and let
  \begin{equation*}
    \cM:=\{f\from\RL\to\R\mid
      \text{$f$~monotone and\/ $\sup\abs{f(\RL)}\le\M$}\}\text.
  \end{equation*}
  Then the following assertions are equivalent:
  \begin{compactenum}[(i)]
    \item\label[statement]{graph}
      For all $J\subseteq\{1,\dotsc,\di\}$, $J\ne\emptyset$,
      the continuous part~$\muc^J$ of the marginal~$\mu^J$ of~$\mu$
      vanishes on every strictly monotone graph~$\graph\subseteq\R^J$:
      \begin{equation*}
        \muc^J(\graph)=0\text.
      \end{equation*}
    \item\label[statement]{heart}
      There exists a set $\Omega'\in\cA$ of full probability $\P(\Omega')=1$
      such that, for all $\omega\in\Omega'$,
      \begin{equation*}
        \sup_{f\in\cM}\abs{\spr f{L_n^{(\omega)}-\mu}}\xto{n\to\infty}0\text.
      \end{equation*}
    \item\label[statement]{wright}
      For all $\epsilon>0$, there are $a=a(\epsilon)>0$ and $b=b(\epsilon)>0$
      such that for all $n\in\N$ there exists an $\Omega_{\epsilon,n}\in\cA$,
      such that for all $\omega\in\Omega_{\epsilon,n}$, we have
      \begin{equation*}
        \sup_{f\in\cM}\abs{\spr f{L_n^{(\omega)}-\mu}}\le\ve
        \qtextq{and\/}
        \P(\Omega_{\epsilon,n})
          \ge1-b\exp(-an)\text.
      \end{equation*}
  \end{compactenum}
\end{Theorem}

\begin{Remark}
  Note that if we knew that the set
  $\{\omega\in\Omega\mid\sup_{f\in\cM}\abs{\spr f{L_n^{(\omega)}-\mu}}\ge\ve\}$
  was measurable, we could rephrase~\labelcref{wright} as follows.
  For all $\epsilon>0$, the probabilities
  $\P(\sup_{f\in\cM}\abs{\spr f{L_n^{(\omega)}-\mu}}\ge\epsilon)$
  converge exponentially fast to zero as $n\to\infty$.
\end{Remark}

We provide a sufficient condition for~\labelcref{graph} in \cref{wright:LDP}
and apply the theorem to our setting.
The idea to use product measures in the context of Glivenko--Cantelli
type theorems appears already in \cite{Stute1976}.

\begin{Theorem}\label{strictlymonotone}
  Let~$\mu$ be a measure on~$\RL$
  which is absolutely continuous with respect to a product measure
  $\Tensor_{j=1}^\di\mu_j$ on~$\RL$, where~$\mu_j$, $j\in\{1,\dotsc,\di\}$
  are measures on~$\R$.
  Then, for each strictly monotone graph $\graph\subseteq\RL$
  we have $\muc(\graph)=0$, where $\muc$
  is the continuous part of~$\mu$.
  Moreover,~\labelcref{graph} from \cref{wright:LDP} is satisfied.
\end{Theorem}
\begin{proof}
  Let~$\density$ be the density of~$\mu$
  with respect to~$\Tensor_{j=1}^\di\mu_j$.
  We define the set of atoms of~$\mu$,
  \begin{equation*}
    S:=\{x\in\RL\mid\mu\{x\}>0\}\textq{ , and}
    S_j:=\{x_j\in\R\mid\mu_j\{x_j\}>0\}
    \quad(j\in\{1,\dotsc,\di\})\text.
  \end{equation*}
  Then we have $S\subseteq S_1\times\dotsb\times S_\di$, and for each
  $x=(x_1,\dotsc, x_\di)\in S_1\times\dotsb\times S_\di$, we have
  \begin{equation}\label{eq:mufprodro}
    \mu\{x\}=\density(x)\prod\nolimits_{j=1}^\di\mu_j\{x_j\}\text.
  \end{equation}
  This implies in particular that for all
  $x\in S_1\times\dotsb\times S_\di\setminus S$, we have $\density(x)=0$.

  In order to prove $\muc(\graph)=0$ it is sufficient to show
    \begin{equation}\label{eq:muGammarho}
      \mu(\graph)
        =\sum_{x\in S\isect\graph}\mu\{x\}\text.
    \end{equation}
  We will prove this by induction over $\di$.
  If $\di=1$ then a strictly monotone graph is a singleton,
  i.\,e.\ $\graph=\{x\}$ for some $x\in\R$.
  Thus,~\eqref{eq:muGammarho} holds true.
  In the case $\di>1$ we assume that~\eqref{eq:muGammarho}
  holds for $\di-1$ and proceed by disintegration.
  Note that, for $z\in\R$, the cross section
  $\graph_z:=\{y\in\R^{\di-1}\mid(y,z)\in\graph\}$
  is itself a strictly monotone graph in~$\R^{\di-1}$.
  Using the cross section $\density_z\from\R^{\di-1}\to\R$,
  $\density_z(y):=\density(y,z)$, $z\in\R$,
  of the density, we define the cross section
  $\mu_z:=\density_z\Tensor_{j=1}^{\di-1}\mu_j$ of the measure~$\mu$.
  By Fubini's Theorem, the disintegration of~$\mu$ is
  \begin{equation*}
    \mu(\d(y,z))
      =\density_z(y)\Tensor_{j=1}^{\di-1}\mu_j(\d y_j)\tensor\mu_\di(\d z)
      \text.
  \end{equation*}
  By the induction hypothesis we obtain
  \begin{equation}\label{eq:muupsi}
    \begin{aligned}
      \mu(\graph)&
        =\int_\R\biggl(\int_{\R^{\di-1}}
                       \ifu{\graph_z}(y)\,\mu_z(\d y)\biggr)\mu_\di(\d z)\\&
        =\int_\R\mu_z(\graph_z)\,\mu_\di(\d z)
        =\int_\R\sum_{y\in\bar S\isect\graph_z}\mu_z\{y\}\,\mu_\di(\d z)\text,
    \end{aligned}
  \end{equation}
  where $\bar S:=S_1\times\dotsb\times S_{\di-1}$.
  The next aim is to show that the set
  $\cZ:=\{z\in\R\mid\bar S\isect\graph_z\ne\emptyset\}$ is countable.
  To this end, we will use that $\bar S$ is countable, define two mappings
  \begin{equation*}
    \phi\from\bar S \to(\bar S\times\R)\isect\graph
    \qtextq{and} \psi\from   (\bar S\times\R)\isect\graph\to\cZ
  \end{equation*}
  and show that they are surjective.
  We first define~$\phi$.
  Let $(y,z),(y,z')\in(\bar S\times\R)\isect\graph$ be given
  and assume without loss of generality that $z\le z'$.
  Let $s\in\{-1,1\}^\di$ be the sign vector of~$\graph$ from \cref{monodef},
  and, again without loss of generality, consider the case $s(\di)=1$.
  Then we have
  \begin{align*}
    \cone_s(y,z)\cap\graph=\{(y,z)\}
    \qtextq{and}
    \cone_s(y,z')\cap\graph=\{(y,z')\}\text.
  \end{align*}
  As $z\le z'$ and $s(\di)=1$, we have $\cone_s(y,z)\supseteq\cone_s(y,z')$,
  such that we obtain
  \begin{align*}
    \{(y,z)\} = \cone_s(y,z)\cap \graph \supseteq \cone_s(y,z')\cap \graph =\{(y,z')\}\text.
  \end{align*}
  This shows that if $y\in\bar S$ is such that there exists an element $z\in\R$
  with $(y,z)\in\graph$, then this $z$ is unique.
  Let $h\in (\bar S\times\R)\isect\graph$ be arbitrary but fixed and set
  \begin{equation*}
   \phi\from\bar S \to(\bar S\times\R)\isect\graph\textq,
   \phi(y)
     :=\begin{cases}
         (y,z) &\text{if $(y,z)\in\graph$, and}\\
          h    &\text{if $(\{y\}\times\R)\isect\graph=\emptyset$.}
     \end{cases}
  \end{equation*}
  This~$\phi$ is well-defined and surjective.
  The mapping~$\psi$ is defined by
  \begin{equation*}
    \psi\from(\bar S\times\R)\isect\graph\to\cZ\textq,\psi(y,z):=z\text.
  \end{equation*}
  To check that~$\psi$ is surjective let $z\in\cZ$ be given.
  Then there exists $y\in \bar S\cap \graph_z$.
  Thus, by definition of~$\graph_z$ we have $(y,z)\in\graph$
  and $(y,z)\in\bar S\times\R$.
  This shows that $(y,z)$ is in the domain of $\psi$ and $\psi(y,z)=z$.

  The surjectivity of~$\phi$ and~$\psi$ and the fact that~$\bar S$ is countable
  show that~$\cZ$ is countable.
  Therefore the last integral in~\eqref{eq:muupsi} is actually a sum:
  \begin{align*}
    \mu(\graph)&
      =\int_\R\sum_{y\in\bar S\isect\graph_z}\mu_z\{y\}\,\mu_\di(\d z)
      =\sum_{z\in S_\di}\sum_{y\in\bar S\isect\graph_z}\mu_z\{y\}\,\mu_\di\{z\}
      =\sum_{x\in S\isect\graph}\mu\{x\}\text.
  \end{align*}
  Here, the last equality follows from~\eqref{eq:mufprodro},
  $\bigcup_{z\in S_\di}(\bar S\cap\graph_z)\times\{z\}\supseteq S\cap\graph$,
  and the fact that~$\density$ vanishes on
  \begin{equation*}
    \bigcup_{z\in S_\di}(\bar S\cap\graph_z)\times\{z\}
      \setminus(S\cap\graph)\subseteq S_1\times\dotsb\times S_k\setminus S
    \text.
  \end{equation*}
  This finishes the induction and we obtained~\eqref{eq:muGammarho}
  and $\muc(\graph)=0$.
  \par
  Let $J\subseteq\{1,\dotsc,k\}$ such that $J\ne\emptyset$
  and $J^c:=\{1,\dotsc,k\}\setminus J\ne\emptyset$.
  Define $\density^J\from\R^J\to\R$ via
  \begin{equation*}
    \density^J(x^J)
     :=\int_{\R^{J^c}}\density(x)\,\d
       \Tensor_{j\in J^c}\mu_j(x^{J^c})\text,
  \end{equation*}
  where $x=(x^J,x^{J^c})\in\R^J\times\R^{J^c}$.
  The function $\density^J$ is the density of the marginal
  $\mu^J$ of~$\mu$ with respect to $\Tensor_{j\in J}\mu_j$,
  since by Fubini for all $A\in\Borel(\R^J)$
  \begin{align*}
    \mu^J(A)
     &=\int_{\RL}\ifu A(x^J)\density(x)\,\d
        \Tensor_{j=1}^\di\mu_j(x)
      =\int_A\density^J(x^J)\,\d\Tensor_{j\in J}\mu_j(x^J)\text.
  \end{align*}
  Thus, the above calculation applies for all marginals of~$\mu$, too.
  This shows~\labelcref{graph} from \cref{wright:LDP}.
\end{proof}

Now we approximate the empirical measure~$L_{m,n}^{r,\omega}$
using the measure~$\P_m^r$, see step~(2) in \cref{sec:wunsch}.
The connection to \cref{defP} is established by \cref{strictlymonotone}.
As announced before we apply the multivariate Glivenko--Cantelli
\cref{wright:LDP} for the proof of \cref{thm:mainGK}.
\begin{Theorem}\label{thm:mainGK}
  Let $\Lambda_n:=\Ico0n\isect\Z^d$, $n\in\N$, a set $\cA\subseteq\R$,
  $\Omega:=\cA^{\Z^d}$, a probability space $(\Omega,\Borel(\Omega),\P)$
  such that~$\P$ satisfies \labelcref{M1,M2,,M3}
  and an admissible function~$f$ be given.
  Besides this let for $m,n\in\N$ and $\omega\in\Omega$
  the empirical measure~$L_{m,n}^{r,\omega}$
  be given as in~\eqref{def:empmeasure:new}
  and let $\P_m^r:=\P_{\Lambda_m^r}$ be the marginal measure,
  where~$r$ is the constant given by~\labelcref{M3}.
  Then there exists a set $\tilde\Omega\in\Borel(\Omega)$ of full measure,
  such that for all $\omega\in\tilde\Omega$ and all $m\in\N$:
  \begin{align}\label{eq:claim:mainGK}
    \lim_{n\to\infty}\norm{\spr{f_m^r}{L_{m,n}^{r,\omega}-\P_m^r}}
      =0\text.
  \end{align}
  Furthermore, for $\K,D,D'>0$ and $r'\in\N$,
  we have for all $\omega\in\tilde\Omega$ and $m\in\N$
  \begin{align}\label{eq:claim:sup}
    \lim_{n\to\infty}\sup_{f\in\cU_{\K,D,D',r'}}
      \norm{\spr{f_m^r}{L_{m,n}^{r,\omega}-\P_m^r}}
      =0\text.
  \end{align}
  Additionally, for each $\epsilon>0$ there exist
  $a=a(\epsilon,m,\K)>0$ and $b=b(\epsilon,m,\K)$ such that for all $n\in\N$
  there is a measurable set $\Omega(\epsilon,n)$ with
  $\P(\Omega(\epsilon,n))\ge1-b\exp(-a \floor{n/m}^d)$ and
  \begin{equation}\label{eq:Wright-applied}
    \sup_{f\in\cU_{\K,D,D',r'}}\norm{\spr{f_m^r}{L_{m,n}^{r,\omega}-\P_m^r}}
      \le\epsilon\qtext{for all $\omega\in\Omega(\epsilon,n)$.}
  \end{equation}
\end{Theorem}
\begin{proof}
  Let $m\in\N$ be given.
  We set $\di:=\setsize{\Lambda_m^r}$
  and embed $\Omega_{\Lambda_m^r}\subseteq\RL$.
  Fix an admissible function~$f$.
  For each $E\in\R$,
  there exists a monotone and bounded function $g_{m,E}^r\from\RL\to\R$
  which extends $f_m^r(\argmt)(E)\from\Omega_{\Lambda_m^r}\to\R$, i.\,e.\ %
  $
    f_m^r(\omega)(E)=g_{m,E}^r(\omega)
  $
  for all $\omega\in\Omega_{\Lambda_m^r}$.
  In fact, $g_{m,E}^r$ can be bounded by $k\K_f$, where~$\K_f$
  is the constant introduced in~\eqref{def:Cf}.
  Thus, the set $\cM_f:=\{g_{m,E}^r\mid E\in\R\}$
  is monotone and bounded by~$k\K_f$, see \cref{rem:admissible}.


In order to apply the Glivenko--Cantelli \cref{wright:LDP},
we enumerate $\Ico0\infty^d\isect m\Z^d$ with a sequence $(t_\ell)_{\ell\in\N}$
such that, for all $q\in\N$,
\begin{equation*}
  \{t_1,\dotsc,t_{q^d}\}
    =\Ico0{mq}^d\isect m\Z^d\text.
\end{equation*}
Consider further for each $\ell\in\N$ the mapping
\[
  X_\ell^r\from\Omega\to\Omega_{\Lambda_m^r}\subseteq\RL,\qquad
  X_\ell^r(\omega):=\Pi_{\Lambda_m^r}(\tau_{t_\ell}^{-1}\omega)\text.
\]
By~\labelcref{M3} the random variables $X_\ell^r$, $\ell\in\N$
are independent with respect to the measure~$\P$ on $(\Omega,\Borel(\Omega))$.
Moreover, applying~\labelcref{M1} shows that $X_\ell^r$, $\ell\in\N$,
are identically distributed.
By definition, the empirical measure of $X_\ell^r$, $\ell\in\{1,\dotsc,\abs{T_{m,n}}\}$,
where $\abs{T_{m,n}}=\lfloor n/m\rfloor^d$,
is exactly the empirical measure $L_{m,n}^{r,\omega}$ given in~\eqref{def:empmeasure:new}.
According to~\labelcref{M2}, the measure~$\P_m^r$ is absolutely continuous
with respect to a product measure on $\Omega_{\Lambda_m^r}$.
We trivially extend $\P_m^r$ and $L_{m,n}^{r,\omega}$ to measures on~$\RL$
(and use the same names for the extensions).
This allows to apply \cref{strictlymonotone},
which gives~\labelcref{graph} of \cref{wright:LDP}.
Thus, the Glivenko--Cantelli theorem implies that
(for the $m\in\N$ chosen above) there is a set $\Omega_m\in\Borel(\Omega)$
of probability one such that for each $\omega\in\Omega_m$ we have
\begin{equation*}
  \bignorm{\spr{f_m^r}{L_{m,n}^{r,\omega}-\P_m^r}}
    =\sup_{g\in\cM_f}\bigabs{\spr g{L_{m,n}^{r,\omega}-\P_m^r}}
    \xto{n\to\infty}0\text,
\end{equation*}
since the supremum is bounded by the supremum in~\labelcref{heart}
from \cref{wright:LDP}.
By the same token,
\begin{equation*}
  \sup_{f\in\cU_{\K,D,D',r'}}\bignorm{\spr{f_m^r}{L_{m,n}^{r,\omega}-\P_m^r}}
    =\sup_{f\in\cU_{\K,D,D',r'}}\sup_{g\in\cM_f}
      \bigabs{\spr g{L_{m,n}^{r,\omega}-\P_m^r}}
    \xto{n\to\infty}0\text.
\end{equation*}
In the light of that, the claimed convergences in~\eqref{eq:claim:mainGK}
and~\eqref{eq:claim:sup} hold independently from~$m$
for all $\omega\in\tilde\Omega:=\bigcap_{m\in\N}\Omega_m$.
To obtain~\eqref{eq:Wright-applied} we apply \cref{wright:LDP},
\labelcref{wright}.
\end{proof}

\section{Almost additivity and limits, Proof of \Cref{thm:main}}\label{sec:Aadd}

Next we investigate the expression $\spr{f_m^r}{\P_m^r}$ for large~$m$.
This is the third and last step in our approximation scheme.
Thus, this step brings us in the position to prove our main results,
namely \cref{thm:main,cor:main}.

\begin{Lemma}\label{lemma3approx}
  Let $\cA\subseteq\R$, $\Omega:=\cA^{\Z^d}$,
  a probability space $(\Omega,\Borel(\Omega),\P)$
  such that~$\P$ satisfies \labelcref{M1,M2,,M3},
  an admissible function~$f$ and the sequence~$(\Lambda_n)$
  with $\Lambda_n=(\Ico0n\isect\Z)^d$, $n\in\N$ be given.
  Besides this, let~$r$ be the constant from~\labelcref{M3}
  and let for $m\in\N$ the marginal measure $\P_m^r:=\P_{\Lambda_m^r}$
  and the function~$f_m^r$ be given as in~\eqref{eq:def:f_n}.
  Then there exists a function $f^*\in\B$ with
  \begin{align}\label{eq:claim:lemma3approx}
    \lim_{m\to\infty}\biggnorm{\frac{\spr{f_m^r}{\P_m^r}}{m^d}-f^*}
      =0\text.
  \end{align}
  Furthermore, we have, with~$b$ and~$D$ from \cref{def:admissible}
  and~$\K_f$ from \cref{rem:admissible},
  for all $m\in\N$
  \begin{equation*}
    \biggnorm{\frac{\spr{f_m^r}{\P_m^r}}{m^d}-f^*}
      \le\frac{b(\Lambda_m^r)}{m^d}
        +(\K_f+D)\frac{\setsize{\partial^r(\Lambda_m)}}{m^d}
      \text.
  \end{equation*}
\end{Lemma}
\begin{proof}
  Let us define $F\from\cF\to\B$ by setting for each $\Lambda\in\cF$:
  \begin{equation*}
    F(\Lambda):=\spr{f_\Lambda}{\P_\Lambda}
      =\int_{\Omega_\Lambda}f_\Lambda(\omega)\,\d\P_\Lambda(\omega)
      =\int_{\Omega}f(\Lambda,\omega)\,\d\P(\omega)\text.
  \end{equation*}
  With this notation, it is sufficient to show that
  $(F(\Lambda_m^r)/m^d)_{m\in\N}$ is a Cauchy sequence.
  \par
  First, we note that~$F$ is translation invariant,
  i.\,e.\ $F(\Lambda + t)=F(\Lambda)$.
  To see this, use~\labelcref{M1}
  and~\labelcref{transitive} of \cref{def:admissible}.
  Note also, that~$F$ is almost additive with the same~$b$ and~$D$ as~$f$,
  see~\labelcref{extensive} of the same definition.
  Furthermore, it follows from \cref{rem:admissible}
  that~$F$ is bounded in the following sense:
  For all $\Lambda\in\cF$, we have $F(\Lambda)\le\K_f\setsize{\Lambda}$
  with the same constant~$\K_f$ as in~\eqref{def:Cf}.
  \par
  Next, assume that two integers $m,M$ with $m\le M$ are given.
  As in~\eqref{def:Tmn}, set
  \begin{align*}
    T_{m,M}:=\{t\in(m\Z)^d\mid\Lambda_m+t
      \subseteq\Lambda_M\}\text.
  \end{align*}
  We are interested in an estimate of the difference
  \begin{align}
    \delta(m,M):=\biggnorm{\frac{F(\Lambda_M^r)}{M^d}
                          -\frac{F(\Lambda_m^r)}{m^d}}\text.
  \end{align}
  To study this we use the triangle inequality and get
  \begin{align}
    \delta(m,M)\le\frac{\alpha(m,M)}{M^d}+\beta(m,M)
  \end{align}
  with
  \begin{equation*}
    \alpha(m,M)
      :=\biggnorm{F(\Lambda_M^r)-\sum_{t\in T_{m,M}}F(\Lambda_m^r +t)}
      \text{, }
    \beta(m,M):=\biggnorm{\frac{F(\Lambda_m^r)}{m^d}-\sum_{t\in T_{m,M}}
                          \frac{F(\Lambda_m^r+t)}{M^d}}
      \text.
  \end{equation*}
  In order to estimate $\alpha(m,M)$, note that
  \begin{equation*}
    \Lambda_M^r
      =\dUnion_{t\in T_{m,M}}(\Lambda_m^r +t)\enspace\dunion\enspace
        \Bigl(\Lambda_M^r\isect\dUnion_{t\in T_{m,M}}
          ((\Lambda_m\isect\partial^r(\Lambda_m))+t)\Bigr)
        \enspace\dunion\enspace
        \Lambda_M^r\setminus(\Lambda_{\floor{M/m}m})\text.
  \end{equation*}
  This and \labelcref{extensive} of \cref{def:admissible} yield
  \begin{align*}&
    \alpha(m,M)\\
    &\le\sum_{t\in T_{m,M}}\Bigl(b(\Lambda_m^r)
      +b\bigl(\Lambda_M^r\isect((\Lambda_m\isect\partial^r(\Lambda_m))+t)\bigr)
      +\bignorm{F\bigl(\Lambda_M^r\isect
        ((\Lambda_m\isect\partial^r(\Lambda_m))+t)\bigr)}\Bigr)
      \\&\qquad
      +b(\Lambda_M^r\setminus\Lambda_{\floor{M/m}m})
      +\norm{F(\Lambda_M^r\setminus\Lambda_{\floor{M/m}m})}\\
    &\le\setsize{T_{m,M}}b(\Lambda_m^r)
      +(\K_f+D)\setsize{T_{m,M}}\setsize{\partial^r (\Lambda_m)}
      +(\K_f+D)\setsize{\Lambda_M^r\setminus\Lambda_{\floor{M/m}m}}
    \text.
  \end{align*}
  Here, we also used translation invariance of~$F$
  and \cref{bounded} of \cref{def:admissible}.
  Dividing this term by~$M^d$ and using $\setsize{T_{m,M}}m^d\le M^d$
  as well as~\eqref{eq:tilingconstruction} leads to
  \begin{align*}
    \frac{\alpha(m,M)}{M^d}
    \le\frac{b(\Lambda_m^r)}{m^d}
      +(\K_f+D)\frac{\setsize{\partial^r(\Lambda_m)}}{m^d}
      +(\K_f+D)\frac{\setsize{\partial^m(\Lambda_M)}}{M^d}\text.
  \end{align*}
  \par
  To estimate~$\beta(m,M)$ we apply again translation invariance of~$F$
  and obtain
  \begin{align*}
    \beta(m,M)
      =\biggnorm{\frac{F(\Lambda_m^r)}{m^d}
                -\frac{\setsize{T_{m,M}}F(\Lambda_m^r)}{M^d}}
      =\biggl(\frac1{m^d}-\frac{\floor{M/m}^d}{M^d}\biggr)
        \norm{F(\Lambda_m^r)}\text.
  \end{align*}
  \par
  Using the properties of the boundary term~$b$,
  the above bounds on~$\alpha(m,M)$ and~$\beta(m,M)$ yield
  \begin{align}\label{eq:limlim}
    \lim_{m\to\infty}\lim_{M\to\infty}\delta(m,M)
      =0\text.
  \end{align}
  This is equivalent to $(F(\Lambda_m^r)/m^d)_{m\in\N}$ being a Cauchy sequence.
  To see this in detail, choose $\epsilon>0$ arbitrarily.
  Then, by~\eqref{eq:limlim}, there exists $m_0\in\N$
  such that $\lim_{M\to\infty}\delta(m_0,M)\le\epsilon/4$.
  Therefore, we find $M_0\in\N$ satisfying $\delta(m_0,M)\le\epsilon/2$
  for all $M\ge M_0$.
  Now, let $j,\di\ge M_0$ be arbitrary.
  Then we obtain using the triangle inequality
  \begin{align*}
    \biggnorm{\frac{F(\Lambda_j^r)}{j^d}-\frac{F(\Lambda_\di^r)}{\di^d}}
    &\le\biggnorm{\frac{F(\Lambda_j^r)}{j^d}-\frac{F(\Lambda_{m_0}^r)}{m_0^d}}
    +\biggnorm{\frac{F(\Lambda_{m_0}^r)}{m_0^d}-\frac{F(\Lambda_\di^r)}{\di^d}}\\
    &=\delta(m_0,j)+\delta(m_0,\di)\le\epsilon\text.
  \end{align*}
  This shows that $(F(\Lambda_m^r)/\setsize{\Lambda_m})_{m\in\N}$
  is a Cauchy sequence and hence convergent in the Banach space~$\B$.
  \par
  Now, that we know that the limit~$f^*$ exists,
  we can study the speed of convergence.
  \begin{multline*}
    \biggnorm{\frac{\spr{f_m^r}{\P_m^r}}{m^d}-f^*}
    =\lim_{M\to\infty}\biggnorm{\frac{\spr{f_m^r}{\P_m^r}}{m^d}
                               -\frac{\spr{f_M^r}{\P_M^r}}{M^d}}
    =\lim_{M\to\infty}\delta(m,M)\\
    \le\lim_{M\to\infty}\Bigl(\frac{\alpha(m,M)}{M^d}+\beta(m,M)\Bigr)
    \le\frac{b(\Lambda_m^r)}{m^d}
      +(\K_f+D)\frac{\setsize{\partial^r(\Lambda_m)}}{m^d}
      \text.\qedhere
  \end{multline*}
\end{proof}

Now we are in the position to prove the main theorem of this paper.
\begin{proof}[Proof of \cref{thm:main,cor:main}]
  The proof is basically a combination of
  \cref{la:empmeasure,thm:mainGK,lemma3approx}.
  We choose $\tilde \Omega$ as in \cref{thm:mainGK},
  $r$ as the constant from~\labelcref{M3}
  and~$f^*\in\B$ according to \cref{lemma3approx}.
  Then we have for arbitrary $m\in\N$ and $\omega\in \tilde \Omega$:
  \begin{multline*}
     \biggnorm{\frac{f(\Lambda_n,\omega)}{n^d}-f^*}
     \le\biggnorm{\frac{f(\Lambda_n,\omega)}{n^d}
    - \frac{\spr{f_m^r}{L_{m,n}^{r,\omega}}}{m^d}}\\
     \quad +\biggnorm{\frac{\spr{f_m^r}{L_{m,n}^{r,\omega}}}{m^d}
    - \frac{\spr{f_m^r}{\P_m^r}}{m^d}}
    +\biggnorm{\frac{\spr{f_m^r}{\P_m^r}}{m^d}-f^*}\text.
  \end{multline*}
  Each of the above mentioned results
  controls one of the error terms on the right hand side,
  which leads to
  \begin{equation*}
    \biggnorm{\frac{f(\Lambda_n,\omega)}{\abs{\Lambda_n}}-f^*}
      \le G(m,n)
        +m^{-d}\bignorm{\spr{f_m^r}{L_{m,n}^{r,\omega}}-\spr{f_m^r}{\P_m^r}}
  \end{equation*}
  with
  \begin{multline*}
    G(m,n):=
      \frac{b(\Lambda_{\floor{n/m}m})}
           {\setsize{\Lambda_{\floor{n/m}m}}}
     +\frac{(2\K+D)\setsize{\partial^m(\Lambda_{n}^m)}}
           {\setsize{\Lambda_n^m}}
     +\frac{2b(\Lambda_m^r)+b(\Lambda_m)+2(\K+D)\setsize{\partial^r(\Lambda_m)}}
           {\setsize{\Lambda_m}}
    \text.
  \end{multline*}
  Taking first the limit $n\to\infty$ and afterwards the limit $m\to\infty$
  on both sides proves \cref{thm:main}.
  \par
  To establish \cref{cor:main}, we use the additional hypotheses on the boundary term and estimate~$G(m,n)$.
  First we note for $n\ge2r$
  \begin{equation*}
    \setsize{\partial^r(\Lambda_n)}
      =(n+2r)^d-(n-2r)^d
      =\sum_{j=0}^d\binom dj(1-(-1)^\di)(2r)^jn^{d-j}
      \le2^{2d+1}r^dn^{d-1}\text.
  \end{equation*}
  Therefore,
  \begin{equation*}
    \frac{b(\Lambda_n)}{\setsize{\Lambda_n}}
      \le\frac{\setsize{\partial^{r'}(\Lambda_n)}D'}{\setsize{\Lambda_n}}
      \le\frac{2^{2d+1}{r'}^dD'}n
  \end{equation*}
  holds for all $n\ge2r'$.
  With $\Lambda_n^m=\Lambda_{n-2m}+(m,m,\dotsc,m)$,
  it is now straightforward to verify
  \begin{equation*}
    G(m,n)
      \le2^{2d+1}\Bigl(\frac{(2\K+D)m^d+D'{r'}^d}{n-2m}
        +\frac{2(\K+D)r^d+3D'{r'}^d}{m-2r}\Bigr)
      \text.
  \end{equation*}
The two claims about
$\sup_{f\in\cU_{\K,D,D',r'}} \bignorm{\spr{f_m^r}{L_{m,n}^{r,\omega}}-\spr{f_m^r}{\P_m^r}}  $ follow from  \cref{thm:mainGK}.
\end{proof}

\section{Eigenvalue counting functions for the Anderson model}\label{sec:Ecf}

In the following, we introduce the Anderson model on~$\Z^d$
or, more precisely, on the graph with nodes~$\Z^d$ and nearest neighbor bonds.
For the corresponding Schr\"odinger operators we show that
the associated eigenvalue counting functions almost surely converge uniformly.

The \emph{Laplace operator} $\Laplace\from\ell^2(\Z^d)\to\ell^2(\Z^d)$
is given by
\begin{equation*}
  (\Laplace\phi)(z)
    =\sum_{x:\fd(x,z)=1}\bigl(\phi(x)-\phi(z)\bigr)
    \qquad(z\in\Z^d)\text.
\end{equation*}
In order to define a random potential,
we introduce the corresponding probability space.
We fix the canonical space $\Omega:=\cA^{\Z^d}$,
where $\cA\subseteq\R$ is an arbitrary subset of $\R$.
As before we equip $\Omega$ with $\Borel(\Omega)$, the $\sigma$-algebra
on $\Omega$ generated by the cylinder sets.
Moreover, we chose a probability measure $\P\from\Borel(\Omega)\to\Icc01$
satisfying \labelcref{M1,M2,,M3}.
In particular, a product measure $\P=\prod_{z\in\Z}\mu$ is allowed,
where $\mu\from\Borel(\cA)\to\Icc01$ is a measure on $(\cA,\Borel(\cA))$.
An alternative way to specify such a product measure is to say
that the projections $\Omega\ni(\omega_x)_{x\in\Z}\to\omega_z$,
$z\in\Z$, are $\cA$-valued i.\,i.\,d.\ random variables.

The \emph{random potential} $V=(V_\omega)_{\omega\in\Omega}$
is now defined by setting for each $\omega=(\omega_z)_{z\in \Z^d}\in\Omega$:
\begin{equation}\label{eq:def:Vomega}
  V_\omega\from\ell^2(\Z^d)\to\ell^2(\Z^d)\textq,
  (V_\omega\phi)(z)=\omega_z\phi(z)
  \qquad(\phi\in\ell^2(\Z^d),z\in\Z^d)\text.
\end{equation}
Together, the Laplace operator and the random potential
form the \emph{random Schr\"odinger operator} $H=(H_\omega)_{\omega\in\Omega}$:
\begin{equation}\label{eq:def:H}
  H_\omega\from\ell^2(\Z^d)\to\ell^2(\Z^d)\textq,
  H_\omega:=-\Laplace+V_\omega\text.
\end{equation}
This operator is almost surely self-adjoint and ergodic by \labelcref{M1,M3}.
Thus, the spectrum~$\sigma(H_\omega)$ of~$H_\omega$
is a non-random subset of~$\R$, cf.~\cite{PasturF-92}.
In the following we are interested in the distribution of~$\sigma(H_\omega)$
on~$\R$.
The function which describes this distribution is called
\emph{integrated density of states}.

Let us define finite dimensional restrictions of~$H$.
To this end, consider for a given set $\Lambda\in\cF$ the projection
\begin{equation}\label{eq:def:proj}
  p_\Lambda\from\ell^2(\Z^d)\to\ell^2(\Lambda)\textq,
  (p_\Lambda\phi)(z):=\phi(z)\textq,
  (\phi\in\ell^2(\Z^d),z\in\Lambda)
\end{equation}
and the inclusion
\begin{equation}\label{eq:def:incl}
  i_\Lambda\from\ell^2(\Lambda)\to\ell^2(\Z^d)\textq,
  (i_\Lambda \phi)(z):=\begin{cases}
    \phi(z)&\text{if $z\in\Lambda$,}\\
    0      &\text{otherwise,}
  \end{cases}
  \quad(\phi\in\ell^2(\Lambda),z\in\Z^d)\text.
\end{equation}
Now,
for any $\omega\in\Omega$ and $\Lambda\in\cF$ we set
\begin{equation*}
  H_\omega^\Lambda\from\ell^2(\Lambda)\to\ell^2(\Lambda)\text, \qquad
  H_\omega^\Lambda:=p_\Lambda H_\omega i_\Lambda\text.
\end{equation*}
The corresponding eigenvalue counting function is given by
\begin{align}\label{eq:def:ecf}
  f(\Lambda,\omega)
    :=\Bigl(\R\ni x\mapsto
        \Tr\bigl(\ifu{\Ioc{-\infty}x}(H_\omega^{\Lambda})\bigr)\Bigr)
  \text.
\end{align}
Here, $\ifu{\Ioc{-\infty}x}(H_\omega^{\Lambda})$
denotes the spectral projection of $H_\omega^\Lambda$ on the interval $\Ioc{-\infty}x$.
Thus, $f(\Lambda,\omega)(x)$ equals the number of eigenvalues
(counted with multiplicities) of $H_\omega^\Lambda$ which do not exceed~$x$.

\begin{Lemma}\label{la:ecf}
  The eigenvalue counting function $f\from\cF\times \Omega\to\B$
  given by~\eqref{eq:def:ecf}
  is admissible in the sense of \cref{def:admissible}.
  It admits a proper boundary term, and possible constants for~$f$ are
  $D=D'=8$, $\K=9$ and $r'=1$.
\end{Lemma}
\begin{proof}
We verify \labelcref{transitive,local,extensive,monotone,bounded} of \cref{def:admissible}.
 \begin{itemize}
  \item[\labelcref{transitive}]
    Let $\Lambda\in\cF$ and $z\in \Z^d$ be given.
Using the definitions of the potential $V_\omega$ in~\eqref{eq:def:Vomega},
the translation $\tau_z$ in~\eqref{eq:def:tau},
the projection $p_\Lambda$ in~\eqref{eq:def:proj}
and the inclusion $i_\Lambda$ in~\eqref{eq:def:incl} we obtain
\begin{equation*}
 p_{\Lambda} V_{\tau_z\omega} i_\Lambda
   = p_{\Lambda +z} V_{\omega} i_{\Lambda +z}\text.
\end{equation*}
This generalizes to $ H_{\tau_z\omega}^\Lambda = H_{\omega}^{\Lambda +z}$
and hence implies for each $x\in \R$
\begin{align*}
 f(\Lambda ,\tau_z\omega) (x)
=\Tr\bigl(\ifu{\Ioc{-\infty}x}(H_{\tau_z\omega}^{\Lambda})\bigr)
=\Tr\bigl(\ifu{\Ioc{-\infty}x}(H_{\omega}^{\Lambda +z})\bigr)
= f(\Lambda +z ,\omega) (x)\text.
\end{align*}
\item[\labelcref{local}]
  Let $\Lambda\in\cF$ be given.
Obviously, for all $\omega,\omega'\in \Omega$
with $\Pi_\Lambda(\omega)=\Pi_\Lambda(\omega')$
we have $H_\omega^\Lambda=H_{\omega'}^\Lambda$.
Thus, we obtain $
 f(\Lambda,\omega)= f(\Lambda,\omega')$.
\item[\labelcref{extensive}]
In order to show almost additivity,
we make use of the following estimate,
which holds for $\Lambda'\subseteq \Lambda\in \cF$ and arbitrary $\omega\in\Omega$:
\begin{align}\label{eq:rankest}
 \norm{f(\Lambda,\omega)- f(\Lambda',\omega)}
   \le4\setsize{\Lambda\setminus \Lambda'}\text.
\end{align}
This bound can be verified using the min-max-principle, cf.~appendix of \cite{LenzSV-10}.
Now let $n\in\N$, disjoint sets $\Lambda_i\in\cF$, $i=1,\dotsc,n$
and $\Lambda:=\Union_{i=1}^n\Lambda_i\in\cF$ be given.
With triangle inequality and~\eqref{eq:rankest} we obtain for each $\omega\in\Omega$:
\begin{align*}
\hspace{-0.3cm}\biggnorm{f(\Lambda,\omega)- \sum_{i=1}^n f(\Lambda_i,\omega)}
&\le\biggnorm{f(\Lambda,\omega)- f\Bigl(\bigcup_{i=1}^n \Lambda_i^1,\omega\Bigr)}
+ \biggnorm{ f\Bigl(\bigcup_{i=1}^n \Lambda_i^1,\omega\Bigr)
- \sum_{i=1}^n f\bigl( \Lambda_i,\omega\bigr)}\\
&\hspace{-2.3cm}\le4\sum_{i=1}^n \setsize{\partial^1(\Lambda_i)}+
 \biggnorm{ f\Bigl(\bigcup_{i=1}^n \Lambda_i^1,\omega\Bigr
)- \sum_{i=1}^n f\bigl( \Lambda_i^1,\omega\bigr)}
+ \sum_{i=1}^n \bignorm{f(\Lambda_i^1,\omega)-f(\Lambda_i,\omega)}\\
&\hspace{-2.3cm}\le8\sum_{i=1}^n\setsize{\partial^1(\Lambda_i)}+
 \biggnorm{ f\Bigl(\bigcup_{i=1}^n \Lambda_i^1,\omega\Bigr)
- \sum_{i=1}^n f\bigl( \Lambda_i^1,\omega\bigr)}\text.
\end{align*}
In order to deal with the last difference,
we use that the operator in consideration has hopping range~$1$,
which gives for $\tilde \Lambda:=\bigcup_{i=1}^n \Lambda_i^1$:
\begin{equation*}
  H_\omega^{\tilde \Lambda}= \bigoplus_{i=1}^n H_\omega^{\Lambda_i^1}\text.
\end{equation*}
Thus, the eigenvalues of $H_\omega^{\tilde \Lambda}$
are exactly the union of the eigenvalues of the operators
$H_\omega^{\Lambda_i^1}$,
$i=1,\dotsc,n$.
This implies
\begin{equation*}
  f(\tilde\Lambda,\omega)
    =\sum_{i=1}^nf(\Lambda_i^1,\omega)
\end{equation*}
and hence
\begin{align*}
  \biggnorm{f(\Lambda,\omega)-\sum_{i=1}^nf(\Lambda_i,\omega)}
    \le8\sum_{i=1}^n\setsize{\partial^1(\Lambda_i)}\text.
\end{align*}
We set $b\colon\cF\to\Ico0\infty$
and $b(\Lambda):=8\setsize{\partial^1(\Lambda)}$.
Let $\Lambda\in\cF$ and $z\in\Z^d$.
Then obviously $b(\Lambda+z)=b(\Lambda)$ and $b(\Lambda)\le8\setsize{\Lambda}$,
and for any sequence of cubes $(\Lambda_n)$ with increasing side length,
we have $b(\Lambda_n)/\setsize{\Lambda_n}\to0$ as $n\to\infty$.
\item[\labelcref{monotone}]
For $\Lambda\in \cF$ and $\omega\in\Omega$ we denote the
$\setsize{\Lambda}$ eigenvalues of $H_\omega^\Lambda$
(counted with multiplicities)
by $E_1(H_\omega^\Lambda)\le\dotsb\le E_{\setsize{\Lambda}}(H_\omega^\Lambda)$.
Choose $n\in\{1,\dotsc,\setsize{\Lambda}\}$ and $\omega\le\omega'$,
i.\,e.\ for each $z\in \Z^d$ we have $\omega_z\le\omega_z'$.
By the  min-max-principle we get for the $n$-th eigenvalue:
\begin{align*}
  E_n(H_\omega^\Lambda)&
    =\min_{\substack{U\subseteq\R^\Lambda,\\\dim(U)=n}}
     \max_{\substack{\phi\in U,\\\norm\phi=1}}
     \spr{H_\omega^\Lambda\phi}\phi\\&
    =\min_{\substack{U\subseteq\R^\Lambda,\\\dim(U)=n}}
     \max_{\substack{\phi\in U,\\\norm\phi=1}}
     \bigl(\spr{H_{\omega'}^\Lambda\phi}\phi
          -\spr{(V_{\omega'}-V_\omega)\phi}\phi\bigr)
    \le E_n(H_{\omega'}^\Lambda)
     \text.
\end{align*}
Therefore, we have for each $x\in\R$ the inequality $
 f(\Lambda,\omega)(x)\ge f(\Lambda,\omega')(x)$.
\item[\labelcref{bounded}]
Let arbitrary $\omega\in \Omega$ be given.
Since the operator $H_\omega^{\lbrace0\rbrace}$
has exactly one eigenvalue, we have $\norm{f(\{0\},\omega)}=1$.
\qedhere
\end{itemize}
\end{proof}

Let us state the main result of this section.
\begin{Theorem}\label{perc:main}
  Let $\Lambda_n:=\Ico0n^d\isect\Z^d$.
  Moreover, let $\cA\subseteq \R$, $\Omega:=\cA^{\Z^d}$
  and $(\Omega,\Borel(\Omega),\P)$
  be a probability space satisfying \labelcref{M1,M2,M3}.
  Consider the random Schr\"odinger operator $H$ defined in~\eqref{eq:def:H}
  and the associated $f$ given in~\eqref{eq:def:ecf}.
  Then there exists a set $\tilde\Omega\in\Borel(\Omega)$ of full measure,
  such that for all $\omega\in\tilde\Omega$:
  \begin{align}
    \lim_{n\to\infty}
      \biggnorm{\frac{f(\Lambda_n,\omega)}{\setsize{\Lambda_n}}-f^*}
      =0\text,
  \end{align}
  where $f^*\in\B$ is given by
  \begin{align}\label{eq:pasturshubin}
    f^*(x)
      :=\E(\spr{\dirac0}{\ifu{\Ioc{-\infty}x}(H_\omega)\dirac0})\text.
  \end{align}
  Here, $\dirac0\in \ell^2(\Z^d)$ is given by $\dirac0(0)=1$
  and $\dirac0(x)=0$ for $x\ne0$.
  Moreover, $\ifu{\Ioc{-\infty}x}(H_\omega)$
  is the spectral projection of~$H_\omega$
  on the interval $\Ioc{-\infty}x$.
  The convergence is quantified by
  \begin{equation}\label{eq:errorestimate}
    \Bignorm{\frac{f(\Lambda_n,\omega)}{\setsize{\Lambda_n}}-f^*}
      \le2^{d+1}\Bigl(\frac{26m^d+8}{n-m}
        +\frac{34r^d+24}{m-r}\Bigr)
        +\sup_{f\in\cU_{\K,D,D',r'}}\frac{\bignorm{\spr{f_{\Lambda_m^r}}
        {L_{m,n}^{r,\omega}-\P_{\Lambda_m^r}}}}{\setsize{\Lambda_m}}
  \end{equation}
  for $n,m\in\N$, $m<n$.
\end{Theorem}

\begin{proof}
 By \cref{la:ecf} we know that the eigenvalue counting function
$f\from\cF\times\Omega\to\B$ is admissible.
Thus we can apply \cref{thm:main} and obtain  that there exists
a function $\bar f\in \B$ and  a set $ \Omega_1\in\Borel(\Omega)$ of full measure
such that for each $\omega\in \Omega_1$ we have
\begin{align}
 \lim_{n\to\infty}\biggnorm{\frac{f(\Lambda_n,\omega)}{\setsize{\Lambda_n}}- \bar f} =0\text.
\end{align}
Thus, it remains to show that $\bar f$ equals the function
\begin{equation*}
  f^*\from\R\to\Icc01\textq,
  f^*(x):=\E(\spr{\delta_{0}}{\chi^\omega_{\Ioc{-\infty}x}\delta_{0}})\text.
\end{equation*}
Here we use ergodicity of~$H$
and infer from \cite[Theorem~4.8]{PasturF-92}
that the there is a set $\Omega_2\in\Borel(\Omega)$
of full measure such that for each $\omega\in\Omega_2$
\begin{align}
  \lim_{n\to\infty}\frac{f(\Lambda_n,\omega)(x)}{\setsize{\Lambda_n}}
    =f^*(x)
\end{align}
for all $x\in\R$ which are continuity points of~$f^*$.
By definition, this is weak convergence of distribution functions.
Thus, as for all $\omega\in\Omega_1\isect\Omega_2$
we have that $f(\Lambda_n,\omega)/\setsize{\Lambda_n}$
converges weakly to $f^*$ and uniformly to $\bar f$, which implies $\bar f=f^*$.
\end{proof}

\begin{Remark}
\begin{asparaitem}
\item The limit~$f^*$ of the normalized eigenvalue counting functions is called
the \emph{integrated density of states} or \emph{spectral distribution function}
of the operator~$H$.
The fact that $f^*$ can be expressed as the function given in
\eqref{eq:pasturshubin} is often referred to as the
\emph{Pastur--Shubin trace formula}, named after the pioneering works
\cite{Pastur-71} and \cite{Shubin-79}.
For more recent results in the specific context we are treating here,
c.\,f.~\cite{Veselic-08,diss:fabian,LenzVeselic2009} and the references therein.
%
%
%

\item Let us also emphasize that the $f^*$ is a deterministic function.
On the one hand this is interesting as this implies that
the normalized eigenvalue counting function converges
for almost all realizations to same limit function.
On the other hand this is not surprising as we mentioned that $H$ is ergodic,
and in this setting it is well-known that the spectrum (as a set)
is deterministic, see for instance \cite{PasturF-92}.
\item The result is easily generalized to sequences of cubes $(\Lambda_n)_n$
of diverging side length with $\Lambda_n\subsetneq\Lambda_{n+1}$.
The validity of the Pastur--Shubin formula shows that the limit~$f^*$
is independent of the specific choice sequence of cubes $(\Lambda_n)_n$.
\item
The statement of \cref{perc:main} has been obtained before in a different setting.
In \cite{LenzMV-08,LenzVeselic2009} ergodic random operators have been considered.
The assumption of ergodicity concerns the measure $\PP$ (in our notation) and is weaker than the
assumptions (M1) to (M3) which we use here. With this regard the result of  \cite{LenzVeselic2009}
is more general than the one obtained here. However, under the mere assumption of ergodicity it is not possible to obtain
explicit error estimates as in~\eqref{eq:errorestimate}.
The paper \cite{LenzMV-08} obtains an error estimate,
similar to, but weaker then~\eqref{eq:errorestimate}.
There the setting is also different from ours here: $\cA$ needs to be countable and
instead of a probability measure properties of frequencies are used.
\item
Similar, but weaker results have been proven for  Anderson-percolation Hamiltonians
in  \cite{Veselic-05b,Veselic-06,LenzVeselic2009}.
These models are particularly interesting
since their integrated density of states exhibits typically an infinite  set of discontinuities,
which lie dense in the spectrum. The random variables entering the Hamiltonian may take uncountably many different values.
\end{asparaitem}
\end{Remark}

\section{Cluster counting functions in percolation theory}\label{sec:ccf}

We introduce briefly percolation on~$\Z^d$.
Percolation comes in two flavors, site and bond percolation.
We focus on site percolation here.
Part of the results have already been obtained in \cite{PogorzelskiS-13}. However, we go far beyond since we not only obtain convergence of densities, but are even able to identify the limit objects.

As before, we let $\Omega:=\R^{\Z^d}$.
We fix the alphabet $\cA:=\{0,1\}$
and a probability measure $\P\from\Borel(\Omega)\to\Icc01$
which is supported in~$\cA^{\Z^d}\in\Borel(\Omega)$,
i.\,e.\ $\P(\cA^{\Z^d})=1$.
A configuration $\omega\in\cA^{\Z^d}\subseteq\Omega$
determines a \emph{percolation graph}
$\Gamma_\omega=(\Z^d,\cE_\omega)$ as follows.
The set of vertices of~$\Gamma_\omega$ is~$\Z^d$,
and an edge connects two vertices if and only if they have distance~$1$
and are both ``switched on''
in the configuration~$\omega=(\omega_z)_{z\in\Z^d}$:
\begin{equation*}
  \cE_\omega:=\bigl\{\{x,y\}\subseteq\Z^d\bigm|
    \fd(x,y)=1,\omega_x=\omega_y=1\bigr\}\text.
\end{equation*}
By this, the percolation graph~$\Gamma_\omega$
is well-defined for $\P$-almost all $\omega\in\Omega$,
and~$\Gamma_\omega$ is a random graph.
For our purposes, we want~$\P$ to satisfy \labelcref{M1,M2,,M3}.
This setting includes but is not limited to the product measure
$\P=\prod_{z\in\Z^d}\mu$, where $\mu\from\Borel(\R)\to\Icc01$
is any probability measure supported on~$\cA$.

We need some standard terminology of graph theory.
Let $\Gamma=(V,\cE)$ be a graph.
For each subset $\Lambda\subseteq V$ of the set of nodes,
$\Gamma$ induces a graph $\Gamma^\Lambda:=(\Lambda,\cE^\Lambda)$ by
\begin{equation*}
  \cE^\Lambda
    :=\{e\in\cE\mid e\subseteq\Lambda\}\text.
\end{equation*}
A \emph{walk} of length~$n\in\N\union\{0,\infty\}$
in the graph~$\Gamma$ is a sequence of nodes
$(z_j)_{j=0}^n\in(\Z^d)^{n+1}$ such that
$\{z_j,z_{j+1}\}$ is an edge of~$\Gamma$, i.\,e.\ %
$\{z_j,z_{j+1}\}\in\cE$, for all $j\in\N\union\{0\}$, $j<n$.
Note that a finite walk of length~$n$ contains~$n$ edges but $n+1$ nodes.
\par
If the walk $(z_j)_{j=0}^n$ has finite length~$n<\infty$,
we say that it \emph{connects} the points~$z_0$ and~$z_n$.
Being connected by a walk is an equivalence relation on the nodes.
We denote the fact that two points $x,y\in\Z^d$
are connected in the graph~$\Gamma$ as~$x\connectedto[\Gamma]y$.

The equivalence classes of~$\connectedto[]$ are called \emph{clusters}.
Let $\Lambda\subseteq\Z^d$ and $x\in\Lambda$.
The cluster of~$x$ in the percolation graph~$\Gamma_\omega^\Lambda$
restricted to~$\Lambda$ consists of all nodes which are connected to~$x$
by a walk in~$\Gamma_\omega^\Lambda$:
\begin{equation*}
  \cluster_x^\Lambda(\omega)
    :=\{y\in\Z^d\mid x\connectedto[\Gamma_\omega^\Lambda]y\}\text,
\end{equation*}
again for $\omega\in\cA^{\Z^d}$, $\Lambda\subseteq\Z^d$ and 
$x,y\in\Lambda$.

\subsection{Convergence of cluster counting functions}\label{sec:per1}
We now define a cumulative counting function for clusters.
As before, let~$\cF$ be the set of finite subsets of~$\Z^d$ and~$\B$
the set of bounded functions from~$\R$ to~$\R$
which are continuous from the right.
The function~$f\from\cF\times\Omega\to\B$
counts the number of clusters in~$\Gamma_\omega^\Lambda$
which are smaller then the given threshold:
\begin{equation}\label{eq:cccf}
  f(\Lambda,\omega)(\lambda)
    :=\bigsetsize{\bigl\{\cluster_z^\Lambda(\omega)\bigm|
      z\in\Lambda,\setsize{\cluster_z^\Lambda(\omega)}\le\lambda\bigr\}}\text.
\end{equation}
Note that~$f$ counts clusters and not vertices in clusters.

\begin{Lemma}\label{la:ccf}
  The cluster counting function $f\from\cF\times\Omega\to\B$
  given by~\eqref{eq:cccf}
  is admissible in the sense of \cref{def:admissible}
  and permits a proper boundary term.
  Possible constants are $D=D'=2$, $r'=1$ and $\K=3$.
\end{Lemma}
\begin{proof}
  We verify \labelcref{transitive,local,extensive,monotone,bounded}
  of \cref{def:admissible}.
 \begin{itemize}
 \item[\labelcref{transitive}]
    Let $\Lambda\in\cF$ and $x,z\in\Z^d$ be given.
    The percolation graph~$\Gamma_\omega$
    is determined by the configuration~$\omega\in\Omega$,
    for almost all $\omega\in\Omega$.
    The shifted configuration gives shifted clusters, i.\,e.
    \begin{equation*}
      \setsize{\cluster_x^\Lambda(\tau_z\omega)}
        =\setsize{\cluster_{x+z}^{\Lambda+z}(\omega)}
    \end{equation*}
    for all $x,z\in\Z^d$.
    Accordingly, for all $\lambda\in\R$,
    \begin{align*}
      f(\Lambda,\tau_z\omega)(\lambda)
     &=\setsize{\bigl\{\cluster_x^\Lambda(\tau_z\omega)\bigm|
        x\in\Lambda,\setsize{\cluster_x^\Lambda(\tau_z\omega)}\le\lambda\bigr\}}\\
     &=\setsize{\bigl\{\cluster_{x+z}^{\Lambda+z}(\omega)\bigm|
        x\in\Lambda,\setsize{\cluster_{x+z}^{\Lambda+z}(\omega)}\le\lambda\bigr\}}
      =f(\Lambda+z,\omega)(\lambda)\text.
    \end{align*}
\item[\labelcref{local}]
  Fix $\Lambda\in\cF$ and $\omega,\omega'\in\Omega$
  with $\omega_\Lambda=\omega'_\Lambda$,
  where $\omega_\Lambda:=(\omega_x)_{x\in\Lambda}$ as before.
  The edges of $\Gamma_\omega^\Lambda$
  are determined by $\omega_\Lambda$.
  Hence, $\Gamma_\omega^\Lambda=\Gamma_{\omega'}^\Lambda$,
  thus $\cluster_z^\Lambda(\omega)=\cluster_z^\Lambda(\omega')$
  for all $z\in\Lambda$, and we obtain
  $f(\Lambda,\omega)=f(\Lambda,\omega')$.
\item[\labelcref{extensive}]
  In order to show almost additivity, fix $\omega\in\Omega$, $n\in\N$
  and disjoint sets $\Lambda_j\in\cF$, $j\in\{1,\dotsc,n\}$.
  We name the union $\Lambda:=\Union_{j=1}^n\Lambda_j$.
  For $x\in\R$,
  \begin{equation*}
    \sum_{j=1}^nf(\Lambda_j,\omega)(x)
  \end{equation*}
  is the total number of clusters of size not larger than~$x$
  in the graphs~$\Gamma_\omega^{\Lambda_j}$.
  Whenever $\fd(\Lambda_j,\Lambda_\di)=1$, the graph~$\Gamma_\omega^\Lambda$
  could contain edges connecting a point in~$\Lambda_j$
  with a point in~$\Lambda_\di$, depending on~$\omega$.
  Each of these edges join two possibly different clusters,
  so for each edge, there are two less small clusters and one more large one.
  By this mechanism, the number of clusters below the threshold~$x$
  changes at most by twice the number of added edges.
  We note
  \begin{equation*}
    \Bigsetsize{\cE_\omega^\Lambda\setminus\Union_{j=1}^n\cE_\omega^{\Lambda_j}}
      \le\sum_{j=1}^n\setsize{\partial^1\Lambda_j}
  \end{equation*}
  and conclude
  \begin{equation*}
    \abs{f(\Lambda,\omega)(x)-\sum_{j=1}^nf(\Lambda_j,\omega)(x)}
      \le2\sum_{j=1}^n\setsize{\partial^1\Lambda_j}
  \end{equation*}
  for all $x\in\R$.
  The choice $b(\Lambda):=2\setsize{\partial^1\Lambda}$ for $\Lambda\in\cF$
  gives a proper boundary term for~$f$, cf.\ \cref{la:ecf}.
\item[\labelcref{monotone}]
  Let $\Lambda\in\cF$ and $\omega,\omega'\in\Omega$, $\omega\le\omega'$.
  Then each edge of~$\Gamma_\omega$ is also an edge in~$\Gamma_{\omega'}$:
  $\cE_\omega\subseteq\cE_{\omega'}$.
  As reasoned in~\labelcref{extensive},
  a new edge never increases the number of clusters below a threshold
  $x\in\R$, so
  \begin{equation*}
    f(\Lambda,\omega)(x)\ge f(\Lambda,\omega')(x)\text.
  \end{equation*}
\item[\labelcref{bounded}]
  For all $\omega\in\Omega$, $f(\{0\},\omega)(x)=0$ for $x<1$ and
  $f(\{0\},\omega)(x)=1$ for $x\ge1$.
  \qedhere
\end{itemize}
\end{proof}

\Cref{thm:main,la:ccf} immediately give the following.
\begin{Corollary}\label{cor:mainPerc}
  Let $\Lambda_n:=\Ico0n^d\isect\Z^d$ for $n\in\N$
  and $f\from\cF\times\Omega\to\B$
  be the cumulative cluster counting function given in~\eqref{eq:cccf}.
  There exists a set $\tilde\Omega\subseteq\Omega$
  of full measure and a function $f^*\in\B$
  such that, for each $\omega\in\tilde\Omega$,
  \begin{equation*}
    \lim_{n\to\infty}\Bignorm{\frac{f(\Lambda_n,\omega)}
                                    {\setsize{\Lambda_n}}-f^*}
      =0\text.
  \end{equation*}
  For all $m,n\in\N$, $m<n$, we have
  \begin{equation*}
    \Bignorm{\frac{f(\Lambda_n,\omega)}{\setsize{\Lambda_n}}-f^*}
    \le2^{d+1}\Bigl(\frac{8m^d+2}{n-m}
      +\frac{10r^d+6}{m-r}\Bigr)
      +\sup_{f\in\cU_{\K,D,D',r'}}\frac{\bignorm{\spr{f_{\Lambda_m^r}}
      {L_{m,n}^{r,\omega}-\P_{\Lambda_m^r}}}}{\setsize{\Lambda_m}}\text.
  \end{equation*}
\end{Corollary}

\subsection{Identification of the limit}

In the previous section we studied the convergence of the counting function in~\eqref{eq:cccf} normalized with $\setsize{\Lambda_n}$.
Next, we give a brief overview on closely related convergence results.
We sketch the proofs only briefly since these results are not in the main focus of this paper.
The heart of the section is that we do not just give statements about convergence, but even present closed expressions of the limits.

We start with defining
\begin{equation}
K_\omega(\Lambda) := \setsize{\{\cluster_x^\Lambda(\omega)\mid x\in \Lambda\}},
\end{equation}
which counts the number of all clusters in $\Gamma_\omega^\Lambda$.
Using this quantity we set:
\begin{align*}
 a_n^{(m)}(\omega)&:=\setsize{\Lambda_n}^{-1} \setsize{\{\cluster_x^{\Lambda_n}(\omega)\mid x \in \Lambda_n, \setsize{\cluster_x^{\Lambda_n}(\omega)}=m\}},\\
 b_n^{(m)}(\omega)&:={K_\omega(\Lambda_n)}^{-1} \setsize{\{\cluster_x^{\Lambda_n}(\omega)\mid x \in \Lambda_n, \setsize{\cluster_x^{\Lambda_n}(\omega)}=m\}},\text{ and}\\
 c_n^{(m)}(\omega)&:=\setsize{\Lambda_n}^{-1} \setsize{\{x \in \Lambda_n \mid  \setsize{\cluster_x^{\Lambda_n}(\omega)}=m\}},
\end{align*}
where again $\Lambda_n:=\Ico0n^d\isect\Z^d$ for $n\in\N$.

\begin{Lemma}\label{la:83}
In the above setting, we have almost surely
 \begin{align*}
  \lim_{n\to\infty} a_n^{(m)}(\omega) &= \frac1m \P(\setsize{\cluster_0}=m),\\
  \lim_{n\to\infty} b_n^{(m)}(\omega) &= \frac{1}{\kappa m} \P(\setsize{\cluster_0}=m),\text{ and}\\
  \lim_{n\to\infty} c_n^{(m)}(\omega) &=  \P(\setsize{\cluster_0}=m),
 \end{align*}
where $\kappa:=\E(\setsize{\cluster_0}^{-1})$.
\end{Lemma}

Note that the existence of the limit in the case corresponding to~$a_n^{(m)}$
was treated in \cref{sec:per1}.
The existence of the limits in \cref{la:83}
has already been proved in \cite{PogorzelskiS-13}
in the setting of bond percolation.
However, the authors did not give explicit expressions for the limit objects.
For the proof of \cref{la:83} one may use \cref{thm:main}
in combination with the $d$-dimensional version of Birkhoff's ergodic theorem,
see \cite{Keller-98}, and the fact \cite{Grimmett-99}
that for almost all $\omega$:
\begin{equation*}
  \lim_{n\to\infty}\frac{K_\omega(\Lambda_n)}{\setsize{\Lambda_n}}=\kappa\text.
\end{equation*}
The above convergence results can again be extend
to the associated distribution functions.
To formulate the corresponding result, we introduce for $n\in\N$ and $\omega\in\Omega$ the maps
$\Theta_\omega^n, \Phi_\omega^n, \Psi_\omega^n:\R\to\R$ by setting for each $m\in\N$
\begin{align*}
 \Theta_\omega^n(m)&:= \sum_{j=1}^{\lfloor m\rfloor } a_n^{(j)}(\omega)
 =\frac{\setsize{\{\cluster_x^{\Lambda_n}(\omega)\mid x\in\Lambda_n, |\cluster_x^{\Lambda_n}(\omega)|\le m\}}}{\abs{\Lambda_n}}\\
\Phi_\omega^n(m)&:= \sum_{j=1}^{\lfloor m\rfloor } b_n^{(j)}(\omega)
 =\frac{\setsize{\{\cluster_x^{\Lambda_n}(\omega)\mid x\in\Lambda_n, |\cluster_x^{\Lambda_n}(\omega)|\le m\}}}{K_\omega(\Lambda_n)}, \text{ and}\\
\Psi_\omega^n(m)&:=\sum_{j=1}^{\lfloor m\rfloor } c_n^{(j)}(\omega)=\frac{\setsize{\{x\in \Lambda_n \mid |\cluster_x(\omega)|\le m\}}}{\abs{\Lambda_n}}.
\end{align*}
Moreover, we define the deterministic functions $\Theta,\Phi, \Psi: \R\to\R$ by
\begin{equation}\label{def:theta}
 \Theta(m):=\sum_{j=1}^{\lfloor m \rfloor} \frac{1}{j}\P(\abs{\cluster_0}=j),\quad
 \Phi(m):=\frac{1}{\kappa}\Theta(m),\quad \text{and}\quad
 \Psi(m):=\P(\setsize{\cluster_0}\le m)
\end{equation}
for $m\in\N$.
\begin{Theorem}\label{thm:perc}
 In the above setting, we can find a set $\tilde\Omega\subseteq\Omega$ of full measure such that for all $\omega\in\tilde \Omega$ we have
\[
 \lim_{n\to\infty}\norm{  \Theta_\omega^n - \Theta  } =0\textq,
\lim_{n\to\infty}\norm{  \Phi_\omega^n - \Phi  }=0
\quad\text{and}\quad
 \lim_{n\to\infty}\norm{  \Psi_\omega^n - \Psi}=0\text.
\]
Here $\norm{\argmt}$ denotes the supremum norm in $\Borel(\R)$.
\end{Theorem}
Let us give a brief sketch of the proof.
The convergence of $\Theta_\omega^n$ and $\Phi_\omega^n$ follows rather direct from \cref{thm:main} and \cref{la:83}. However, in order to obtain the convergence of $\Psi_\omega^n$ one has to apply a different scheme,
which was used in the context of the eigenvalue counting function in
\cite[Section~6]{LenzVeselic2009}.
The strategy  consist of the following steps:
One first verifies weak convergence of the distribution functions and second, shows that $\nu_\omega^n(\{\lambda\})\to \nu(\{\lambda\})$ for each $\lambda\in\R$. Here $\nu$ and $\nu_\omega^n$ are the measures corresponding to $\Psi$ and $\Psi_\omega^n$, respectively.
Both steps together imply uniform convergence.  To verify these convergences one applies again \cref{la:83} as well as Birkhoff's ergodic theorem.

\begin{Remark}
  The first statement of \cref{thm:perc} identifies the limit~$f^*$
  from \cref{cor:mainPerc}, namely it shows $f^*=\Theta$,
  where~$\Theta$ is given in~\eqref{def:theta}.
\end{Remark}

\appendix

\section{Examples of measures}\label{example:density}

Let us discuss three classes of examples of measures~$\P$
satisfying \labelcref{M1,M2,,M3}.
\begin{enumerate}[(a)]
\item \emph{Countable colors:}
  Consider the case $d=1$ and let $\cA=\N_0$.
  Let $\Omega=\R^\Z$ and fix an arbitrary product measure
  $\tilde\P\from\Borel(\Omega)\to\Icc01$ with support
  \begin{equation*}
    \supp\tilde\P\subseteq\cA^\Z\text.
  \end{equation*}
  We define a transformation of~$\tilde\P$.
  To this end, let constants
  $c,\beta,\alpha_{-c},\alpha_{-c+1},\dotsc,\alpha_{c}\in\N_0$
  be given and consider the function
  \[
    \phi\from\Omega\to\Omega,\qquad (\phi(\omega))_z
     :=\beta+ \sum_{\di=-c}^c \alpha_\di \omega_{z-\di}\text.
  \]
  We define $\P:=\tilde \P\circ \phi^{-1}$.
  Let us check the conditions \labelcref{M1,M2,,M3} for~$\P$.
  In order to check~\labelcref{M1} let $z\in\Z$ be given.
  Then, using stationarity of the product measure~$\tilde\P$,
  \begin{align*}
   \P\circ\tau_z^{-1}
    &=\tilde\P\circ\phi^{-1}\circ\tau_z^{-1}
     =\tilde\P\circ(\tau_z\circ\phi)^{-1}\\
    &=\tilde\P\circ(\phi\circ\tau_z)^{-1}
     =\tilde\P\circ\tau_z^{-1}\circ\phi^{-1}
     =\tilde\P\circ\phi^{-1}
     =\P\text.
  \end{align*}

  Let us verify condition~\labelcref{M2} for~$\P$.
  We define for each $\Lambda\in\cF$ the function
  \[
    \rho_\Lambda\from\cA^\Lambda\to \R,\qquad x
      =(x_z)_{z\in \Lambda}\mapsto\rho_\Lambda(x)
      =\P(\Pi_\Lambda^{-1}(\{x\}))\text.
  \]
  Then $\rho_\Lambda$ is the density of the marginal measure~$\P_\Lambda$
  with respect to the counting measure on~$\N_0$,
  since we have for each $\Lambda\in \cF$ and $A\in \Borel(\cA^\Lambda)$
  \[
   \P_\Lambda(A)
    =\sum_{x\in A} \P(\Pi_\Lambda^{-1}(\{x\}))
    =\sum_{x\in A}\rho_\Lambda(x)\text.
  \]

  It remains to verify condition~\labelcref{M3}.
  To this end, let $\Lambda_1,\dotsc,\Lambda_n\subseteq \Z$
  with $\min\{\fd(\Lambda_i,\Lambda_j)\mid i\neq j\}>2c$ be given.
  Then, using the definition of $\phi$, we have for each
  $x=(x_z)_{z\in\Lambda}\in\Lambda:=\bigcup_{i=1}^n\Lambda_i$
  \begin{align*}
   \P(\Pi_\Lambda^{-1}(\{x\}))
    =(\tilde\P\circ\phi^{-1}\circ\Pi_\Lambda^{-1})(\{x\})
    =\prod_{i=1}^n(\tilde\P\circ\phi^{-1}\circ\Pi_{\Lambda_i}^{-1})(\{x\})\text,
  \end{align*}
  which proves that $\rho_\Lambda= \prod_{i=1}^n\rho_{\Lambda_i}$.

\item \emph{Normal distribution:}
Here, we treat the case $d=1$, $\cA=\R$, $\Omega=\R^\Z$ and set
$
\tilde \P:=\bigotimes_{z\in \Z} \cN(0,1)\from\Borel(\Omega)\to\Icc01\text,
$
where $\cN(0,1)$ is the standard normal distribution.
For $c\in\N_0$ and $\beta,\alpha_{-c},\alpha_{-c+1},\dotsc,\alpha_{c}\in\R$
we use
\[
 \phi\from\Omega\to\Omega,\qquad (\phi(\omega))_z
   =\beta+ \sum_{k=-c}^c \alpha_k \omega_{z-k}
\]
to define $\P:=\tilde \P\circ \phi^{-1}$.
As before, the conditions \labelcref{M1,M3}
are implied by the choice of $\phi$ and the product structure of~$\tilde\P$.
For~\labelcref{M2}, let $\Lambda\subseteq \Z$ be finite
and first assume that $\Lambda=\Icc ab\cap \Z$, $a,b\in\Z$.
We define the matrix
\[
 A_{\Lambda}\in \R^{\Lambda\times \{a-c,\dotsc, b+c\}}\text,
\qquad (A_\Lambda)_{i,j}=\alpha_{i-j}\text,
\]
where $\alpha_k:=0$ if $k\notin\{-c,\dotsc, c\}$.
Recall that
$\P_\Lambda
=\tilde \P\circ \phi^{-1} \circ \Pi_\Lambda^{-1}
=\tilde \P\circ  (\Pi_\Lambda \circ\phi)^{-1}$.
For $\omega\in\Omega$
we get
\[
 \Pi_\Lambda(\phi(\omega)) = A_\Lambda \Pi_{\Icc{a-c}{b+c}}(\omega) + \beta e_\Lambda,
\]
where $e_\Lambda=(1,\dotsc,1)^\top \in \R^{\Lambda}$.
Now, it follows that $\P_\Lambda$ is normal distributed
with mean $\beta e_\Lambda$ and covariance matrix $A_\Lambda A_\Lambda^\top$.
Note that $A_\Lambda A_\Lambda^\top$ is invertible
since the rows of $A_\Lambda$ are linearly independent.
Thus, the measure $\P_\Lambda$ is absolutely continuous
with respect to the multi-dimensional Lebesgue measure.

In the situation where $\Lambda$ is not of the form $\Icc ab\isect \Z$,
consider the interval $I:=\Icc{\min \Lambda}{\max\Lambda}\isect\Z$.
The measure $\P_\Lambda$ is a marginal measure of  $\P_I$ and therefore has a density.
\item \emph{Abstract densities and finite range:}
In the following we develop a more general example with densities.
Again, we consider for simplicity reasons the case $d=1$,
however this is easily generalized to higher dimensions.
Choose $\cA,\cB\in\Borel(\R)$ and independent
$\cB$-valued random variables~$X'_x$, $x\in\Z$ with density $g\from\cA\to\R_+$.
We use the abbreviation $X_{\Icc m\ell}:=(X_m,\dotsc,X_\ell)$.
We utilize a function $\phi\from\cB^{k+1}\to\cA$
to introduce the $\cA$-valued random variables
\begin{equation*}
  X_x:=\phi(X'_{\Icc x{x+k}})\qtext{$x\in\Z$.}
\end{equation*}
We require from~$\phi$, that there is a function
$\psi\from\cA\times\cB^k\to\cB$ such that
\begin{equation*}
  \psi(\phi(x_{\Icc0k}),x_{\Icc1k})=x_0
\end{equation*}
for all $x_{\Icc0k}\in\cB^{k+1}$.
Further,~$\psi$
shall be continuously differentiable w.\,r.\,t.\ its first argument:
$\psi':=D_1\psi$.
An example of such a pair of functions is
\begin{equation*}
  \phi(x_{\Icc0k}):=\frac1{k+1}\sum_{j=0}^kx_j\textq,
  \psi(\xi_0,x_{\Icc1k}):=(k+1)\xi_0-\sum_{j=1}^kx_j\text,
\end{equation*}
where $\cA:=\cB:=\Icc01$ and $\psi'(\xi_0,x_{\Icc1k})=k+1$.
In this example, $(X_x)_x$ is a moving average process.
By suitable modifications,
all moving average processes are seen to be included in our setting.

\begin{Proposition}
  Fix a finite set $\Lambda\subseteq\Z$.
  Under the specified circumstances,
  the joint distribution of $(X_x)_{x\in\Lambda}$,
  is absolutely continuous with respect to Lebesgue measure on~$\cA^\Lambda$.
\end{Proposition}
\begin{proof}
  Without loss of generality, we treat only the case
  $\Lambda=\{1,\dotsc,\ell\}$.
  By construction, for $A_1,\dotsc,A_\ell\subseteq\cA$ measurable,
  \begin{align*}
    p\etdef\P(X_1\in A_1,\dotsc,X_\ell\in A_\ell)\\
     &=\int_{\cB^{\ell+k}}\dx_{\Icc1{\ell+k}}
        \prod_{m=1}^\ell\ifu{A_m}(\phi(x_{\Icc m{m+k}}))\cdot
        \prod_{m=1}^{\ell+k}g(x_m)\text.
  \end{align*}
  By Fubini and induction on $j\in\{0,\dotsc,\ell\}$, we see
  \begin{multline*}
    p=\int_{A_1\times\dotsb\times A_j}\d\xi_{\Icc1j}\int_{\cB^{\ell-j+k}}\dx_{\Icc{j+1}{\ell+k}}\\
        \prod_{m=1}^j\bigl(g(\psi(\tilde x_m^{(j)}))\abs{\psi'(\tilde x_m^{(j)})}\bigr)\cdot
        \prod_{m=j+1}^\ell\ifu{A_m}(\phi(x_{\Icc m{m+k}}))\cdot
        \prod_{m=j+1}^{\ell+k}g(x_m)\text,
  \end{multline*}
  where $\tilde x_j^{(j)}:=(\xi_j,x_{\Icc{j+1}{j+k}})$ and, for $1\le m<j$, the term
  $\tilde x_m^{(j)}:=\tilde x_m^{(j-1)}|_{x_j\mapsto\psi(\xi_j,x_{\Icc{j+1}{j+k}})}$
  is generated from $\tilde x_m^{(j-1)}$ by substituting~$x_j$
  by $\psi(\xi_j,x_{\Icc{j+1}{j+k}})$.
 For the induction step use the substitution
$\xi_j:=\phi(x_{\Icc j{j+k}})$ or $x_j=\psi(\xi_j,x_{\Icc{j+1}{j+k}})$ in
  \begin{align*}
    &\int_\cB \dx_j\ifu{A_j}(\phi(x_{\Icc j{j+k}}))f_j(x_j)g(x_j)\\
    &=\int_{A_j}\d\xi_j f_j(\psi(\xi_j,x_{\Icc{j+1}{j+k}}))g(\psi(\xi_j,x_{\Icc{j+1}{j+k}}))
      \abs{\psi'(\xi_j,x_{\Icc{j+1}{j+k}})}\text,
  \end{align*}
 for any $x_{\Icc{j+1}{j+k}}\in\cB^k$ and suitable $f_j\from\cB\to\R_+$.
 For $j=\ell$, we conclude
  \begin{align*}
    p&=\int_{A_1\times\dotsb\times A_\ell}\d\xi_{\Icc1\ell}\int_{\cB^k}\dx_{\Icc{\ell+1}{\ell+k}}
        \prod_{m=1}^\ell\bigl(g(\psi(\tilde x_m^{(\ell)}))\abs{\psi'(\tilde x_m^{(\ell)})}\bigr)
        \cdot\prod_{m=\ell+1}^{\ell+k}g(x_m)
        \text.
  \end{align*}
  We hereby identified the density with respect
to the product Lebesgue measure on~$\cA^\ell$.
\end{proof}
\end{enumerate}

\def\polhk#1{\setbox0=\hbox{#1}{\ooalign{\hidewidth
  \lower1.5ex\hbox{`}\hidewidth\crcr\unhbox0}}}

\end{document}


\providecommand{\bysame}{\leavevmode\hbox to3em{\hrulefill}\thinspace}
\providecommand{\MR}{\relax\ifhmode\unskip\space\fi MR }
\providecommand{\MRhref}[2]{%
  \href{http://www.ams.org/mathscinet-getitem?mr=#1}{#2}
}
\providecommand{\href}[2]{#2}

%

\end{document}